\documentclass{amsart}
\usepackage{bbm}
\usepackage{lipsum}
\usepackage{enumerate}
\usepackage{showlabels}
\usepackage{amsmath,amsthm,amsfonts,amssymb,amscd,mathrsfs}
\usepackage[pdftex, colorlinks=true]{hyperref}
\newtheorem*{Mheo}{Main Theorem}
\newtheorem{theo}{Theorem}[section]

\newtheorem{lemma}{Lemma}[section]
\newtheorem{cor}{Corollary}[section]
\newtheorem{defi}{Definition}[section]
\newtheorem{ques}{Question}[section]
\newtheorem{rem}{Remark}[section]
\newtheorem*{Claim}{Claim}
\newtheorem{CClaim}{Claim}
\newtheorem*{Fact}{Fact}
\usepackage{graphicx}
 \usepackage[all]{xy}
 
\usepackage{xcolor}
\usepackage{boxhandler}
\usepackage{caption}
\usepackage{lipsum}
\def \ex {\mathsf{ex}}

\def \Id {\mathsf{Identity}}

\def \Per {\mathsf{Per}}
\def \pper{\mathsf{per}}
\def \ap{\mathsf{ap}}
\def \dim{\mathsf{dim}}
\def \htop{\mathsf{h_{top}}}
\def \deg {\mathsf{deg}}
\def \com {\mathsf{comp}}
\def \Res {\mathsf{Res}}
\def \Exp {\mathsf{exp}}
\begin{document}


\begin{large}
\title{Periodic expansiveness of smooth surface diffeomorphisms and applications}

\author{DAVID BURGUET}
\address{LPMA - CNRS UMR 7599 \\
\newline
Universite Paris 6 \\
 75252 Paris Cedex 05 FRANCE\\
 \newline
david.burguet@upmc.fr}


\begin{abstract}
We prove that periodic asymptotic expansiveness introduced in \cite{em} implies the equidistribution of periodic points to   measures of maximal entropy. Then following Yomdin's approach \cite{Yom} we show by using semi-algebraic tools  that $C^\infty$ interval maps and $C^\infty$ surface diffeomorphisms satisfy this expansiveness property respectively for repelling and saddle hyperbolic points with Lyapunov exponents uniformly away from zero.
\end{abstract}



\maketitle
\section{Introduction}
In this paper we establish new relations between the entropy and the growth of periodic points of smooth dynamical systems.
 In general the topological entropy may not be equal to the  exponential growth in $n$ of $n$-periodic points. Of course it may be less, e.g. when the system has uncountably many periodic points, but it may also be larger : there are even  minimal (therefore aperiodic) smooth systems with positive entropy \cite{Her}. However these two quantities coincide in many cases, in particular under some expansive and specification properties. Then one may also wonder if periodic points are equidistributed with respect to measures of maximal entropy.

 For expansive homeomorphisms with the standard specification property R. Bowen   proved  in \cite{bow} the equality between the entropy and the growth of periodic points, together with the equidistribution of periodic points with respect to  the unique invariant measure of maximal entropy. In particular these properties hold true for topologically transitive subshifts of finite type and Axiom A systems \cite{boww}. Also by a  celebrated result of A. Katok \cite{Kat} any  $C^{1+\alpha}$ surface diffeomorphism admits  hyperbolic horseshoes with entropy arbitrarily close to the topological entropy so that the exponential  growth in $n$ of saddle (hyperbolic) $n$-periodic points  is larger than or equal to the topological entropy. \\

  Here we will prove the following  result: 	\begin{Mheo}
Let $f:M\rightarrow M$ be a $C^\infty$ diffeomorphism of a compact surface $M$ (resp. a $C^\infty$ interval map) with positive topological entropy $\htop(f)>0$.

Then for any $\delta\in ]0,\htop(f)[$  the set $\Per_{n}^\delta$ of saddle (resp. repelling) $n$-periodic points with Lyapunov exponents $\delta$-away from zero grows exponentially in $n$ as the topological entropy. Moreover these periodic points are equidistributed with respect to measures of maximal entropy, i.e.:

\begin{itemize}
\item $ \limsup_{n\rightarrow +\infty}\frac{1}{n}\log \sharp \Per_{n}^\delta=\htop(f), $
\item  for all increasing sequences of positive integers $(n_k)_k$ satisfying $$ \lim_{k\rightarrow +\infty}\frac{1}{n_k}\log \sharp \Per_{n_k}^\delta=\htop(f),$$
any weak-star limit $\mu$ of the sequence $\left(\frac{1}{\sharp \Per_{n_k}^\delta}\sum_{x\in \Per_{n_k}^\delta}\delta_x\right)_{k }$ is an $f$-invariant measure of maximal entropy, i.e.:
$$h(\mu)=\htop(f).$$ \\
\end{itemize}
\end{Mheo}

We list now some comments and earlier related works enlightening the above statement:\\

\textit{\textsf{i. Periodic points with small Lyapunov exponents.}} For $1\leq r\leq \infty$  V. Kaloshin \cite{Kal} has proved that  in Newhouse domains (i.e. $C^r$  open sets with a dense subset of diffeomorphisms having an homoclinic tangency) generic $C^r$ smooth  surfaces diffeomorphisms have arbitrarily fast growth of saddle periodic points (see \cite{asa} for analytic examples). Kaloshin stated his result for finite $r$, but his proof  also works for $r=\infty$. For these $C^\infty$ surface diffeomorphisms we can therefore not replace
the sets $\Per_n^\delta$ with the sets $\Per_n$ of all $n$-periodic points in our Main Theorem. In fact, it follows from the Main Theorem that  any  family of saddle periodic points of  a surface diffeomorphism  with exponential growth larger than the topological entropy contains periodic points with an arbitrarily small Lyapunov exponent. Finally we  also recall that by Kupka-Smale theorem \cite{K,S} all periodic points of a $C^r$ generic diffeomorphism are hyperbolic (in particular the sets $\Per_n$ are finite for all positive integers $n$).\\

\textsf{\textit{ii. Systems with Lyapunov exponents uniformly away from zero.} }When all invariant measures are hyperbolic with Lyapunov exponents uniformly away from zero then one may consider in the Main Theorem the  set $\Per_n$ of $n$-periodic points (because  in this case we have $\Per_n^\delta=\Per_n$ for some positive $\delta$). Such nonuniformly hyperbolic smooth surface diffeomorphisms have been built and studied in \cite{Cao,Ta}. More recently P. Berger proved that Henon like diffeomorphisms satisfy this property and our result recover then the equidistribution property shown in \cite{Ber}. \\

\textsf{\textit{iii. Comparison with a weaker form due to Chung and Hirayama.}} A version of the above theorem was proved for $C^{1+\alpha}$ surface diffeomorphisms $f$ in  \cite{ch} (see also \cite{Gel}), where the authors considered for some fixed $c,\delta>0$ 	and  for any positive integer $n$ the subset $\Per^n_\delta(c)$ of $x\in \Per_\delta^n$ with   $\|T_{f^ix}f^k\|, \|(T_{f^ix}f^k)^{-1}\|\geq ce^{k \delta }$ for all $k\in \mathbb{N}$ and for all $0\leq i <n$. The union $\bigcup_n\Per^n_\delta(c)$ is the set of  periodic points of its closure, say $K_{\delta,c}:=\overline{\bigcup_n\Per^n_\delta(c)}$. Moreover $K_{\delta,c}$ is a uniformly hyperbolic  set. According to  the  aforementioned pioneer work of R. Bowen the exponential growth in $n$ of $\Per_n^\delta(c)$ coincides with the topological entropy restricted to $K_{\delta,c}$. Then by  Katok's Theorem one concludes that  the exponential growth in $n$ of $\Per_n^\delta(c)$ goes to $\htop(f)$ when $c$ and $\delta$ go to zero.  Under $C^\infty$ smoothness assumption our Main Theorem shows the stronger fact that
 the exponential growth in $n$ of $\bigcup_{c>0}\Per_n^{\delta}(c)$ is equal to the topological entropy. By using Bowen's result Chung and Hirayama also proved that periodic points in $\Per_n^\delta(c)$ are equidistributed with respect to an invariant measure $\mu_c$ which is converging to a measure of maximal entropy when $c$ goes to zero.    \\

\textsf{ \textit{iv. The case of topological Markov shifts.}} A similar picture occurs in the non compact setting of connected topological Markov shifts with a countable  set of states.
Indeed the (Gurevic's) entropy of such a Markov shift, which may be defined as the supremum of the measure theoretic  entropies of invariant ergodic probability measures, coincides with  the supremum of the topological entropy of  finite subgraphs. As for a $C^{1+\alpha}$ surface diffeomorphism the system thus contains a family of subshifts of finite type  with entropy arbitrarily close to the entropy of the  system. Although  periodic points are equidistributed with respect to  measures of maximal entropy for these finite subgraphs, this  is false in general for the topological Markov shift. By a result of Vere-Jones \cite{ver}  it holds true for the so called \textit{strongly positive recurrent} topological Markov shifts (see  Proposition 2.3 in \cite{verb} for some (equivalent) definitions of strongly positive recurrence).\\

\textsf{\textit{v. Positive recurrence in Sarig's Markov representation.}} The above results must be compared with the recent work of O. Sarig  about the coding of $C^{1+\alpha}$ surface diffeomorphisms.
In \cite{sar} he built  for any $\delta>0$ a finite to one topological Markov shift extension of any surface diffeomorphism on a set of full measure for any hyperbolic ergodic measure with Lyapunov exponents  $\delta$-away from zero. Then the finite subgraphs of this Markov shift correspond to the hyperbolic sets. However it is unknown   if the connected components of this Markov shift are strongly positive recurrent even for $C^\infty$ surface diffeomorphisms.\\

\textsf{\textit{vi. Existence and  finiteness of ergodic measures of maximal entropy.}}
Building on Yomdin's theory, S. Newhouse proved that any $C^\infty$ map on  a compact manifold admits an ergodic  measure of maximal entropy \cite{new} (see also \cite{Buz}).  J. Buzzi  proved also in \cite{Buz} that there were only finitely many such measures for $C^\infty$ interval maps with positive topological entropy. It was very recently shown   by using  Sarig's Markov representation that it also holds true for surface diffeomorphisms \cite{OBC}.
 In our Main Theorem  we do not know how to  identify the limits $\mu$ of periodic measures in
$\Per^\delta$ (in general the measures $\mu$ are not ergodic). For topologically transitive such systems,   uniqueness of the measure of maximal entropy has been established in \cite{Buz} \cite{OBC} so that  periodic points in $\Per^\delta$ are  actually equidistributed with respect to this measure. \\

\textsf{\textit{vii. Lower bound on the growth of periodic points.}} By applying Gurevic's theory to the Markov representations of Sarig \cite{sar} and Buzzi  \cite{Buz} for respectively smooth surface diffeomorphisms and interval maps, these dynamical systems when they admit a measure of maximal entropy (in particular in the $C^\infty$ case) satisfy  the following estimate\footnote{In the present paper $\mathbb{N}:=\{0,1,...,\}$ will denote the set of nonnegative integers.}:
$$\exists p\in \mathbb{N}\setminus \{0\}, \ \liminf_{n\rightarrow +\infty, \ p\mid n }e^{-n \htop(f)} \sharp \Per_n>0.$$
We will prove in the Appendices that we can  replace in the above inequality the set $\Per_n$ by $\Per_n^{\delta}$ for any positive $\delta$ less than the topological entropy. This follows from the construction of Sarig in the first case whereas for interval maps we have to slightly modify the Markov representation of Buzzi. In particular we get the following statement.

\begin{cor}
With the assumptions and notations  of the Main Theorem,  there exists a positive integer $p$ such that :
\begin{itemize}
\item $ \lim_{n\rightarrow +\infty, \ p\mid n }\frac{1}{n}\log \sharp \Per_{n}^\delta=\htop(f), $
\item  any weak-star limit of $\left(\frac{1}{\sharp \Per_{n}^\delta}\sum_{x\in \Per_{n}^\delta}\delta_x\right)_{n, \ p\mid n }$ is a measure of maximal entropy.  \\
\end{itemize}
\end{cor}

\textsf{\textit{viii. Open question.}}
Does the Main Theorem also hold true for $C^r$ surface diffeomorphims with  finite $r>1$ provided $\htop(f)> \Lambda^+(f)/r$ where $\Lambda^+(f)$ denotes the supremum over all invariant measures of the sum of the positive Lyapunov exponents? \\

The proof of the Main Theorem follows from the following general strategy: any subset of periodic points whose growth is larger than  or equal to the topological entropy, but with zero local exponential growth, is equidistributed with respect to  maximal measures. Moreover its exponential growth rate is in fact equal to the topological entropy.  By local exponential growth we mean the exponential growth in $n$ of the $n$-periodic points lying in a arbitrarily small $n$-dynamical ball. From Katok's Theorem we then only have to check that
$(\Per_n^\delta)_n$ has zero local exponential growth rate for a $C^\infty$-surface diffeomorphism (or a $C^\infty$ interval map). In \cite{em} \cite{bdo} such systems are said to be asymptotically periodic-expansive.
This condition was used there  to build so-called uniform generators. In our setting we get:

\begin{cor}\label{sm}
Let $f:M\rightarrow M$ be a $C^\infty$ surface diffeomorphism (or interval map).

Then for any $\delta>0$   there exists a uniform $\delta$-generator $P$, that is a finite measurable partition $P$ of  the complement of $\bigcup_{n\in \mathbb{N}\setminus \{0\}}\Per_n \setminus \Per_n ^\delta$ such that the maximal diameter of the elements of 
$ \bigvee_{k=-n}^n f^{k}P$ goes to zero when $n$ goes to infinity. 

Moreover we may assume that with  $\pper^\delta(f):=\sup_{n\in \mathbb{N}\setminus \{0\}}\frac{\log \sharp \Per_n^\delta(f)}{n}$ the partition  $P$ satisfies  $$\sharp P \leq e^{\max\left(\htop(f),\pper^\delta(f)\right)}+1.$$
\end{cor}

When  the surface diffeomorphism is nonuniformly hyperbolic with Lyapunov exponents uniformly from zero (as discussed in \textsf{\textit{ii.}}), then for some positive $\delta$ the uniform $\delta$-generator is covering  the whole surface.

\section{Local periodic growth, the abstract framework}

Let $(X,T)$ be a topological dynamical system, i.e. $X$ is a compact metrizable space and $T$ is a continuous map from $X$ into itself. We denote by $\Per(T)$ (or $\Per$) the set of periodic points of $(X,T)$. When $\mathcal{P}$ is an invariant subset of $\Per(T)$ and $n$ is a positive integer we let $\mathcal{P}_n$ be the subset  of $n$-periodic points in $\mathcal{P}$, that is $$\mathcal{P}_n:=\{ x\in \mathcal{P}, \ T^nx=x\}.$$

For any invariant subset $\mathcal{P}\subset \Per$ we let $$g_\mathcal{P}=\limsup_n\frac{1}{n}\log\sharp\mathcal{P}_n.$$

 We also introduce the following  local quantity. For any $\epsilon>0$ we denote by  $B_T(x,n,\epsilon)$  (or just $B(x,n,\epsilon)$) the $n$-dynamical ball at $x$ of size $\epsilon$ :

$$B_T(x,n,\epsilon)=\bigcap_{0\leq k<n}T^{-k}B(T^kx,\epsilon).$$

 Then we let

$$g_\mathcal{P}^*(\epsilon):=\limsup_n\frac{1}{n}\sup_{x\in X}\log\sharp\left( \mathcal{P}_n\cap B(x,n,\epsilon)\right)$$

and

$$g_\mathcal{P}^*=\lim_{\epsilon \rightarrow 0}g_\mathcal{P}^*(\epsilon).$$

The local periodic growth $g_\mathcal{P}^*$ is defined in a similar way to the tail entropy $h^*$ introduced by M. Misiurewicz in \cite{Miss} and somehow  represents for the periodic growth what the tail entropy does for the topological entropy. In general $h^*$ and $g^*_{\Per}$ may differ for the same reasons as for the topological entropy and the (global) periodic growth. For systems with the standard specification property\footnote{For a continuous interval map this property is satisfied if and only if the map is topologically mixing \cite{Bl}.} we always have  $g_{\Per}\geq \htop$  and $g_{\Per}^*\geq h^*$ \cite{den}. \\

We may also define  $g_\mathcal{P}^*$ by using the infinite dynamical ball $$B(x, \infty,\epsilon):=\bigcap_{\alpha< k<+\infty}T^{-k}B(T^kx,\epsilon),$$
with $\alpha=-\infty$ it $T$ is invertible and $\alpha=-1$ if not. Observe that when $p$  belongs to $\mathcal{P}_n\cap B(x,n,\epsilon)$ then we have
$$\mathcal{P}_n\cap B(x,n,\epsilon)\subset B(p,  \infty,2\epsilon).$$
Consequently we have:

\begin{lemma}
$$g_\mathcal{P}^*=\lim_{\epsilon\rightarrow 0} \limsup_n\frac{1}{n}\sup_{x\in X}\log \sharp \left( \mathcal{P}_n\cap B(x, \infty,\epsilon)\right).$$
\end{lemma}
Alternatively we may also take the supremum over $x\in \mathcal{P}_n$ on the right hand side.\\

Recall a topological system $(X,T)$ is  \textit{expansive} when there is $\epsilon>0$ so that for all $x\in X$ the infinite dynamical ball  $B(x, \infty,\epsilon)$ is reduced to the singleton $\{x\}$. The system will be said to be \textit{asymptotically $\mathcal{P}$-expansive} when  $g_\mathcal{P}^*=0$. Aperiodic systems as well as expansive ones are obviously asymptotically $\mathcal{P}$-expansive for any $\mathcal{P}\subset \Per$. \\

Clearly if $(X,T)$ is asymptotically $\mathcal{P}$-expansive with finite topological entropy then the growth in $n$ of $n$-periodic points of $\mathcal{P}$ is at most exponential and we have more precisely:

\begin{lemma}
$$g_\mathcal{P}\leq \htop(T)+g_\mathcal{P}^*.$$
\end{lemma}

We give now a lower bound for the entropy of any equidistributed  measure  for $\mathcal{P}$. For any $n$ we let $\nu_n^\mathcal{P}$ be the atomic  measure given by $\nu_n^\mathcal{P}:=\frac{1}{\sharp \mathcal{P}_n}\sum_{x\in \mathcal{P}_n}\delta_x$, where $\delta_x$ is the Dirac measure at $x$. Clearly $\nu_n^\mathcal{P}$ is an invariant probability measure.

\begin{lemma}\label{cru} Let $(X,T)$ and $\mathcal{P}$ be as above. Assume $\sharp \mathcal{P}_n<\infty$ for all $n$. Then
for any weak-star limit  $\mu$ of $(\nu_{n_k}^\mathcal{P})_k$, where $(n_k)_k$ is an increasing  sequence of integers with
$ g_\mathcal{P}=\lim_k\frac{1}{n_k} \log \sharp \mathcal{P}_{n_k}$, we have

\begin{eqnarray*}
h(\mu)\geq g_\mathcal{P}- g_\mathcal{P}^*.
\end{eqnarray*}

\end{lemma}

Before the proof we recall some basic definitions and facts about entropy in ergodic theory \cite{Dowb}. For a Borel probability measure $\mu$ on $X$ \footnote{Of course  all these notions may be defined in a more general context, but we will always consider topological dynamical systems on compact metrizable spaces in the present paper.} and a finite Borel partition $P$ of $X$  we let  $H_\mu(P)$  be the Shannon entropy of $\mu$ with respect to $P$:
$$H_\mu(P):=-\sum_{A\in P}\mu(A)\log \mu(A).$$
By concavity of $x\mapsto -x\log x$ on $[0,1]$ we always have \begin{eqnarray}\label{concc}H_\mu(P)&\leq &\log \sharp P.\end{eqnarray} Observe also that the map $\nu\mapsto H_\nu(P)$ defined on the (compact) set of Borel probability measures endowed  with the weak-star topology  is continuous at $\mu$ whenever $\mu(\partial A)=0$ for every $A\in P$.
When $Q$ is another finite Borel partition we may define the conditional quantity $H_\mu(P|Q)$ as follows:
$$H_\mu(P|Q):=-\sum_{B\in Q}\mu(B)H_{\mu^B}(P),$$
where $\mu^B$ is the conditional  measure on $B$ defined by  $\mu^B(C)=\frac{\mu(B\cap C)}{\mu(B)}$ for any Borel set $C$. The following inequality is well-known:
\begin{eqnarray}\label{cond}H_\mu(P)&\leq & H_\mu(Q)+H_\mu(P|Q).
\end{eqnarray}

If  $\mu$ is moreover $T$-invariant the sequence $\left(\frac{H_\mu(P^n)}{n}\right)_n$, with $P^n=\bigvee_{k=0}^{n-1}T^{-k}P$, is nonincreasing
 and its limit is usually denoted by $h_\mu(P)$. Then the measure theoretic (or Kolmogorov-Sinai) entropy
$h(\mu)$ of $\mu$ is defined as the supremum of $h_\mu(P)$ over all finite Borel partitions $P$. When $(P_k)_k$ is a sequence of finite partitions whose diameter is going to zero then one can show (see  \cite{Dowb}) that:
$$h(\mu)=\lim _k h_\mu(P_k).$$

Let us go back now to the proof of Lemma \ref{cru}.

\begin{proof}
Let $\mu$ be the limit of a subsequence  $(\nu_{n_k}^\mathcal{P})_k$ and let $\epsilon>0$. Let $P$ be a partition of diameter less than $\epsilon$ with $\mu(\partial A)=0$ for all $A\in P$.
For any large $n$ we let $Q_n$ be a finite Borel partition of $X$ finer than $P^n$ such that any element of $Q_n$ contains at most one point of $ \mathcal{P}_n$.

For any invariant measure $\nu$, the sequence $\left(\frac{H_\nu(P^n)}{n}\right)_n$ is nonincreasing. Therefore  we have for any $0<l\leq n_k$ (in what follows $\nu_{n_k}:=\nu_{n_k}^\mathcal{P}$):

\begin{eqnarray*}
\frac{1}{l}H_{\nu_{n_k}}(P^l) &\geq &\frac{1}{n_k}H_{\nu_{n_k}}(P^{n_k}).
\end{eqnarray*}
From (\ref{cond}) and (\ref{concc}) we get then:
\begin{eqnarray*}
\frac{1}{l}H_{\nu_{n_k}}(P^l)& \geq &   \frac{1}{n_k}\left(H_{\nu_{n_k}}(Q_{n_k})- H_{\nu_{n_k}}(Q_{n_k}|P^{n_k})\right),\\
& \geq & \frac{1}{n_k}\left(H_{\nu_{n_k}}(Q_{n_k})- \sum_{A\in P^{n_k}} \nu_{n_k}(A)H_{\nu^A_{n_k}}(Q_{n_k})\right),\\
\frac{1}{l}H_{\nu_{n_k}}(P^l)  & \geq &   \frac{1}{n_k}\left(\log \sharp \mathcal{P}_{n_k}-\sup_{A\in P^{n_k}} \log\sharp \left(\mathcal{P}_{n_k}\cap A\right)\right).
\end{eqnarray*}
As the diameter of any $A\in P$  is less than $\epsilon$ we have
$$\frac{1}{l}H_{\nu_{n_k}}(P^l)   \geq    \frac{1}{n_k}\left(\log \sharp \mathcal{P}_{n_k}-\sup_{x\in X} \log\sharp \left(\mathcal{P}_{n_k}\cap B(x,n_k,\epsilon)\right)\right).$$
By continuity the left term is going to  $\frac{1}{l}H_{\mu}(P^l)$ when $k$ goes to infinity, whereas for the right term we have
\begin{eqnarray*}\liminf_k\frac{1}{n_k}\left(\log \sharp \mathcal{P}_{n_k}-\sup_{x\in X} \log\sharp \left(\mathcal{P}_{n_k}\cap B(x,n_k,\epsilon)\right)\right)&\geq & \lim_k\frac{1}{n_k}\log \sharp \mathcal{P}_{n_k}- \\
&  &\limsup_k\frac{1}{n_k}\sup_{x\in X} \log\sharp \left(\mathcal{P}_{n_k}\cap B(x,n_k,\epsilon)\right),\\
&\geq &  g_\mathcal{P}-g^*_{\mathcal{P}}(\epsilon).
\end{eqnarray*}
Thus we get for all positive integer $l$:
$$\frac{1}{l}H_{\mu}(P^l)\geq g_\mathcal{P}-g^*_{\mathcal{P}}(\epsilon),$$
and by taking the limit in $l$:
$$h_\mu(P)\geq g_\mathcal{P}-g^*_{\mathcal{P}}(\epsilon).$$
The above inequality holds for any finite partition $P$ whose elements have diameter less than $\epsilon$ and boundaries with zero $\mu$-measure. Since these partitions generate the whole Borel sigma-algebra up to sets of  zero $\mu$-measure we get therefore:
$$h(\mu)\geq g_\mathcal{P}-g^*_{\mathcal{P}}(\epsilon).$$
Finally we conclude  the proof by taking the limit when $\epsilon$ goes to zero.
\end{proof}

By the variational principle \cite{go,goo} the topological entropy is equal to the supremum of the measure theoretic entropy of all invariant measures. A measure whose entropy attains this supremum (and therefore  is equal to the topological entropy) is said to be maximal or of maximal entropy. Such measures do not always exist (even in intermediate smoothness, i.e. for  $C^r$ smooth systems with $r<+\infty$ \cite{Mis3}).  The following corollary gives a new criterion for the existence of maximal measures and the equidistribution of periodic points with respect to these measures beyond Bowen's result:

\begin{cor}\label{conc}
Assume $T$ is asymptotically $\mathcal{P}$-expansive and $g_\mathcal{P}\geq \htop(T)$. Then we have
\begin{itemize}
\item $g_\mathcal{P}=\htop(T)$,
\item any weak-star limit of $(\nu_{n_k}^\mathcal{P})_k$ with
$ g_\mathcal{P}=\lim_k\frac{1}{n_k} \sharp \mathcal{P}_{n_k}$ is a measure of maximal entropy.
\end{itemize}
\end{cor}

For a  topological system $(X,T)$ with the small boundary property a measure theoretic analogue of $g_\mathcal{P}^*$ was defined in \cite{bdo}. A system is said to have the small boundary property\footnote{It was proved in \cite{Lin} that on a finite dimensional manifold any dynamical system with countable periodic points have this property.} when it admits a basis of neighborhoods with null boundary, i.e. with null $\mu$-measure for any invariant measure $\mu$. Let $(\epsilon_k)_k$ be a decreasing sequence  of real numbers going to zero. Then for any invariant measure $\mu$ we let
\begin{eqnarray}\label{tto} g^*_\mathcal{P}(\mu)&=& \lim_k \limsup_{\nu_n\xrightarrow{n\rightarrow\infty} \mu}\mathfrak{P}_k(\nu_n)\end{eqnarray}
with $$\mathfrak{P}_k(\nu_n)= \frac{1}{n}\int\log \sharp \left(\mathcal{P}_n\cap B(x,n,\epsilon_k)\right )d\nu_n(x),$$ where $\nu_n$ are  periodic measures associated to a periodic point in  $\mathcal{P}_n$ with minimal period equal to $n$.
When there is no such sequence $(\nu_n)_n$ converging to $\mu$ we let  $g^*_\mathcal{P}(\mu)=0$.
 Then we have the following variational principle (see Appendix B):

\begin{eqnarray}\label{varia}g_\mathcal{P}^*=\sup_\mu g^*_\mathcal{P}(\mu).
\end{eqnarray}

\begin{rem}
In the statement of Lemma \ref{cru}  one can replace $g_\mathcal{P}^*$ by $g^*_\mathcal{P}(\mu)$. The easy proof is left as an exercise for the reader.
\end{rem}

\section{Local periodic growth for $C^r$, $r>1$,  interval maps and surface diffeomorphisms}

We investigate here the local periodic growth for $C^r$  smooth systems $f:M\rightarrow M$ for a real\footnote{The map $f$ is $C^r$ if it is $\lfloor r \rfloor$ times differentiable and its $\lfloor r \rfloor$-differential map is $r-\lfloor r \rfloor$-Holder.} number $r\geq 1$ on a compact manifold $M$.
It is known that for  $C^r$  generic diffeomorphims in dimension larger than one  the growth of hyperbolic  periodic points is superexponential  in Newhouse domains \cite{Kal} and thus these diffeomorphisms have infinite local exponential  growth $g_\mathcal{P}$ with respect to $\mathcal{P}=\Per$. Similar explicit examples have been also built on the interval \cite{kk}. When $f :M\rightarrow M$ is a surface diffeomorphism (resp. a $C^1$ interval map)  we let $\Per^\delta$ for $\delta>0$ be the set of saddle (resp. repelling) periodic points of $f$ with Lyapunov exponents  $\delta$-away from zero. For any $f$-invariant measure $\mu$ we let $\chi^+(\mu,f)$ be the integral with respect to $\mu$ of the positive part of the  largest Lyapunov exponent, i.e. $\chi^+(\mu,f):=\int \lim_n \frac{1}{n}\log^+ \|T_xf^n\|d\mu(x)$.
We also denote by $R(f)=\sup_\mu\chi^+(\mu,f)$ its  supremum over all invariant measures, which can be written as $R(f):=\lim_n\frac{1}{n}\log^+ \sup_{x\in X}\|T_xf^n\|$ (see \cite{b}).

\begin{theo}\label{mainn}
Let $f$ be a $C^r$  interval (or circle) map with $r>1$. Then for any $\delta>0$ and any invariant measure $\mu$
we have
$$g_{\Per^\delta} ^*(\mu)\leq \frac{\chi^+(\mu,f)}{r-1},$$
in particular $$g_{\Per^\delta} ^*\leq \frac{R(f)}{r-1}.$$
\end{theo}

For surface diffeomorphisms we will show the following upper bound:

\begin{theo}\label{key}
Let $f:M\rightarrow M$ be a $C^r$ surface diffeomorphism with $r>1$. Then for any $\delta>0$ and any invariant measure $\mu$
we have
$$g_{\Per^\delta} ^*(\mu)\leq \frac{2(3r-1)}{r(r-1)}\max\left(\chi^+(\mu,f),\chi^+(\mu,f^{-1})\right),$$
in particular $$g_{\Per^\delta} ^*\leq \frac{2(3r-1)}{r(r-1)} \max\left(R(f),R(f^{-1})\right).$$
\end{theo}

The examples built in \cite{Bub} show that the upper bound in Theorem \ref{mainn} for $C^r$ interval maps with finite $r>1$ is sharp. But we do not think this is the case for surface diffeomorphisms in Theorem \ref{key}.
Let us now give some consequences of the above theorems. First recall we have by Katok's theorem  \cite{Kat} and its noninvertible version \cite{cc,y}:

\begin{theo}\label{Kat}
Let $f$ be a $C^r$ surface diffeomorphism (or interval map) with $r>1$ and let $\mu$ be an invariant measure with positive entropy. Then for any $0<\delta< h(\mu)$:

$$g_{\Per^\delta} \geq h(\mu),$$

in particular for any $\delta<\htop(f)$, $$g_{\Per^\delta}\geq \htop(f).$$
\end{theo}

Therefore any $C^\infty$ surface diffeomorphism or interval map $f$ is asymptotically $\Per^\delta$-expansive and satisfies  $g_{\Per^\delta} \geq \htop(f)$ for any $0<\delta<\htop(f)$. Stronger lower bounds follow from the Markov representations built  in \cite{sar},\cite{Buz} and from Gurevic's theory:

\begin{theo}\cite{sar}\cite{Buz}(See Appendix A)
Let $f$ be a $C^r$ surface diffeomorpshim (or interval map) with $r>1$ admitting a measure of maximal entropy. Then for any $0<\delta< h_{top}$ there is a positive integer $p$ such that
$$\liminf_{n\rightarrow +\infty, \ p\mid n}e^{-n\htop(f)}\sharp \Per_n^\delta>0.$$
\end{theo}

The main statements of  the Introduction  follow then directly  from  Corollary  \ref{conc}.

\begin{ques}Is any $C^\infty$ diffeomorphism $f:M\rightarrow M$ with $\dim(M)\geq 3$  asymptotically $\Per^\delta$-expansive for all $0<\delta<\htop(f)$?
\end{ques}

A symbolic extension is a topological extension by a subshift over a finite alphabet (a symbolic extension does not need to be Markovian and the extension does not have to be finite-to-one).  Existence of symbolic extensions for smooth systems has been widely studied. Such extensions are  known for $C^\infty$  smooth systems  but also  for $C^r$ (with $r>1$) interval and surface maps (it is still open in higher dimension for finite $r>1$). In \cite{bdo} the authors refine the abstract theory of symbolic extensions developed in \cite{BD} to build uniform generators (see Theorem \ref{sm} in the Introduction).  This is done not only  by controlling the entropy
at different scales as in \cite{BD} but also but also by using the local periodic growth, captured by the map $\mu\mapsto g_{\mathcal{P}}^*(\mu)$. Uniform generators are directly related with symbolic extensions with a Borel embedding (see Theorem 1.2 in \cite{bdo}). 
In this context we get  with Theorems \ref{mainn} and \ref{key}:

\begin{cor}\label{ccc}(See Appendix C)
Let $(M,f)$ be a $C^r$ surface diffeomorphism, resp. a $C^r$ interval map, with $r>1$. 

Then for any $\delta>0$, there is a symbolic extension
$\pi:(Y,S)\rightarrow (M,f)$ and a Borel embedding $\psi:B\rightarrow Y$ with $B$ a Borel set with full measure for every ergodic measure except  periodic measures in $\bigcup_n \left( \Per_n\setminus \Per_n^\delta\right)$ such that $\psi\circ f=\sigma\circ \psi$ and $\pi\circ \psi=\Id_B$ and

$$\sup_{\nu,  \ S^*\nu=\nu \text{ and } \pi^*\nu=\mu}h(\nu)= E(\mu),$$

where for any for any $f$-invariant measure $\mu$ we let:
$$E(\mu):=h(\mu)+\frac{2(3r-1)}{r(r-1)}\left(\chi^+(\mu,f)+\chi^+(\mu,f^{-1})\right),$$

$$ \text{resp. }  E(\mu)=  h(\mu)+\frac{\chi^+(\mu,f)}{r-1}.$$

Moreover the cardinality of the alphabet of $Y$ may be chosen to be less than or equal to $e^{ \max \left( \sup_{\mu}E(\mu),\pper^\delta \right) }+1$ with  $\pper^\delta:=\sup_{n\in \mathbb{N}\setminus \{0\}}\frac{\log \sharp \Per_n^\delta}{n}$.
\end{cor}

Corollary \ref{sm} stated in the introduction follows then from Theorem 1 in \cite{bdo} by taking $r=\infty$. Indeed, in this case the alphabet of $Y$ may be chosen less than or equal to $e^{\max(\htop(f),\pper^\delta)}+1$ and a uniform $\delta$-generator is given by $P:=\psi^{-1}Q$ with $Q$ being the zero-coordinate partition of $Y$ (for a proof we refer to Theorem 1.2 in \cite{bdo} which relates uniform generators and  symbolic extensions with a Borel embedding).\\

The upper bounds on the local periodic growth obtained above in Theorems  \ref{mainn} and  \ref{key} are reminiscent of the following ones obtained by J. Buzzi in \cite{Buz} for the tail entropy of a $C^r$ map $f$ on a compact manifold of dimension $d$: $$h^*(f)\leq \frac{d R(f)}{r}.$$
In fact, the proof of both estimates uses semi-algebraic tools introduced by Y.Yomdin in \cite{Yom}.


\section{The case of interval maps}
The proof of Theorem \ref{mainn} will follow quite directly from the following reparametrization lemma of dynamical balls proved  in \cite{Burguet}. For a $\mathcal{C}^r$ (with $r>1$) smooth map $f$, we let $\|f\|_r:=\max_{s=1,...,[r],r}\|f^{(s)}\|_\infty$ where $\|f^{(s)}\|_\infty$  is the usual supremum norm of the $s$-derivative of $f$ for $s\leq [r]$ (where $[r]$ denotes the integer part of $r$) and $\|f^{r}\|_\infty$ is the $r-[r]$ Holder norm of $f^{([r])}$. We also let $H:[1,+\infty[\rightarrow \mathbb{R}$ be the function defined for all $t\geq 1$ by $H(t)=-\frac{1}{t}\log(\frac{1}{t})-(1-\frac{1}{t})\log (1-\frac{1}{t})$.

\begin{lemma}\label{main}\cite{Burguet}
Let $f:\mathbb{S}^1\rightarrow \mathbb{S}^1$ be a $\mathcal{C}^r$ circle map with $r>1$ and let $\delta>0$.  Then   for any positive integer $p$ there is $\epsilon>0$,  depending only on  $\|f^p\|_r$, such that for all  $x\in  \mathbb{S}^1$ and  for all positive integers $n$ there exists a finite collection of intervals $\mathcal{J}_n(x)$ with the following properties  :

\begin{enumerate}[(i)]
\item $\forall J\in \mathcal{J}_n(x), \ \forall y,z \in J, \ |(f^n)'(y)-(f^n)'(z)|\leq \frac{1}{3}\sup_{t\in J} |(f^n)'(t)|,$
\item $\left\{y\in \mathbb{S}^1, \ |(f^n)'(y)|\geq e^{\delta n} \right\}\cap B(x,n,\epsilon)\subset \bigcup_{J\in \mathcal{J}_n(x)}J,$
\item $\log\sharp \mathcal{J}_n(x) \leq  [n/p]\left(\frac{1}{r-1}+H([\lambda^+_n(x,f^p)-p\delta]+3)\right)\left(\lambda^+_n(x,f^p)-p\delta\right)+[n/p]A+B,$
\end{enumerate}
where $A$ is a constant depending only on $r$, whereas $B$ is a constant depending only on $f$ and $p$ (not on $n$). Also we have used the notation  $$\lambda^+_n(x,f^p):=\frac{1}{[n/p]}\sum_{j=0}^{[n/p]-1} \log^+|(f^p)'(f^{jp}x)|.$$

\end{lemma}

\begin{proof}[Proof of Theorem \ref{mainn}]For a fixed $\delta>0$ we will  in fact prove the following stronger upper bound  for any $f$-invariant measure $\mu$: $$g_{\Per^\delta}^*(\mu)\leq \frac{\chi^+(\mu)-\delta}{r-1}.$$ Let $\mu$ be an $f$-invariant measure and let $0<\gamma\ll 1$. We consider an integer $p=p(\mu,\gamma,r)$ so large that $\frac{1}{p}\int \log^+|(f^p)'| d\mu\simeq \chi^+(f,\mu)$, $\frac{A}{p}\ll 1$ (with $A=A(r)$ as in Lemma \ref{main}) and $H(p\gamma) \ll 1$.   We apply Lemma \ref{main} for $f$, $p$ and $\delta$ and we let $\epsilon>0$ be as in the statement.  It follows from the first item there is at most one point of $\Per^\delta_n$ in any $J\in \mathcal{J}_n(x)$ as $f^n$ is expanding on any such interval $J$ intersecting $\Per_n^\delta$ (once one takes $n$ with $e^{n\delta}>\frac{4}{3}$). Then for any  periodic measure $\nu_n$ supported on the orbit of a periodic point $x\in \Per_n^\delta$ with minimal period $n$ we have whenever $\epsilon_k$ is less than $\epsilon$:
$$\mathfrak{P}_k(\nu_n)= \frac{1}{n}\int\log \sharp \left(\mathcal{P}_n\cap B(x,n,\epsilon_k)\right )d\nu_n(x)\leq\frac{1}{n} \int \log\sharp \mathcal{J}_n(x)d\nu_n(x).$$  If  $\lambda^+_n(x,f^p)-p\delta\geq p\gamma$ we have $H([\lambda^+_n(x,f^p)-p\delta]+3))\left(\lambda^+_n(x,f^p)-p\delta\right)<<\lambda^+_n(x,f^p)-p\delta$ and if not,  $H([\lambda^+_n(x,f^p)-p\delta]+3)\left(\lambda^+_n(x,f^p)-p\delta\right)\leq p\gamma\log 2 $. Thus for $n>>\max(B,p)$ with $B$ as in Lemma \ref{main} we get 
\begin{eqnarray*}
\frac{1}{n } \int \log\sharp \mathcal{J}_n(x)d\nu_n(x)&\lesssim &\frac{1}{ p(r-1)} \int \left(\lambda^+_n(x,f^p)-p\delta\right)d\nu_n(x),\\
&\lesssim &\frac{1}{ p(r-1)} \frac{1}{[n/p]}\sum_{j=0}^{[n/p]-1} \int \left(\log^+|(f^p)'(f^{jp}x)|-p\delta\right)d\nu_n(x).
\end{eqnarray*}
As all terms in this last sum are equal we have
$$\frac{1}{n } \int \log\sharp \mathcal{J}_n(x)d\nu_n(x) \lesssim \frac{1}{ p(r-1)} \int \left(\log^+|(f^p)'(x) |-p\delta\right)d\nu_n(x).$$
 When $\nu_n$ is going to $\mu$ then this last term is going by continuity of the integrand to $$\frac{1}{p(r-1)}\left(\int \log^+|(f^p)'|d\mu-p\delta\right)\simeq \frac{\chi^+(f,\mu)-\delta}{r-1}.$$ We conclude that
$g_{\Per^\delta}^*(\mu)=\lim_k\limsup_{\nu_n\rightarrow \mu} \mathfrak{P}_k(\nu_n)\leq \frac{\chi^+(f,\mu)-\delta}{r-1}.$
\end{proof}

\section{$n$-hyperbolic hexagons}
In the next sections we consider a $C^r$ surface diffeomorphism with $r>1$. The proof of the Main Theorem for such a diffeomorphism follows the same strategy as the above one dimensional case. In  Lemma \ref{main} we reparametrized dynamical balls by intervals with bounded distortion so that they contain each at most one repelling periodic point. In dimension two we introduce now local hyperbolic hexagons which play somehow the same role as these intervals. \\

We first define a notion of "finite time local hyperbolic sets" involving cones and  we  recall some related terminology and notations.
\subsection{Cones}
Let $(E,\|\cdot\|)$ be a two-dimensional Euclidean vector space and $v,w$ be two  (nonzero) noncolinear vectors in $E$. The associated cone $C(v,w)$ of $E$ is the subset of the vector space $E$ given by  $C(v,w)=\{\alpha v+\beta w, \ \alpha\beta \geq 0\}$.

 The \textit{aperture} of the cone $C(v,w)$, which will be denoted by  $\ap(C(v,w))$, is the unsigned angle $\angle (v,w) \in ]0,\pi[$ between $v$ and $w$ and the center of $C(v,w)$ is the line $\mathbb{R}(v+w)$.  We also refer to the two unit vectors generating this line as oriented centers of the cone. A subset $C$ of $E$ is said to be a cone when there are $v,w\in E$ with $C=C(v,w)$.
For $\alpha>0$  and a  cone $C$ with aperture less than $\frac{\pi}{\alpha}$  we let $C(\alpha)$ be the cone with the same center as $C$ but with $\ap(C(\alpha))=\alpha \ap(C)$. 
Finally two cones $C, C'$ are said transverse when $C\cap C'=\{0\}$ and $\alpha$-transverse for $\alpha>0$ when for any $(u,u')\in C\times C'$ the unsigned angle between $u$ and $u'$ is in $]\alpha, \pi-\alpha[$. A regular $C^1$ curve $\lambda:[0,1]\rightarrow E$ is said tangent to a cone $C$ when $0\neq \lambda'(t) \in C$ for all $t\in [0,1]$.

For any $0<\alpha<\pi$ we fix  $\mathcal{C}_\alpha$  the family of cones of $\mathbb{R}^2$ (endowed with the usual Euclidean norm) with aperture equal to $\alpha$ centered at the lines given by the $[4/\alpha]+1$-roots of the unity in $\{z\in \mathbb{C}, \ Re(z)\geq 0\} \simeq \mathbb{R}\times \mathbb{R}^+\subset \mathbb{R}^2$. Clearly the cones $C$ in $\mathcal{C}_\alpha$ but also the cones $C(1/2)$ for $C\in \mathcal{C}_\alpha$ are covering $\mathbb{R}^2$ and the cardinality of $\mathcal{C}_\alpha$ is less than $4/\alpha+1$. Finally we let $\mathfrak{C}_\alpha$ be the collection of pairs of cones - later called bicones -   $(C,C')\in \mathcal{C}_\alpha\times \mathcal{C}_\alpha$ given by $\alpha$-transverses cones.


\subsection{ $n$-hyperbolic set}\label{ffh}
Let $n$ be a fixed nonegative integer. We consider a sequence $(E_k,\|\cdot \|_k)_{k=0,...,n}$ of two-dimensional Euclidean spaces with $(E_0,\|\cdot\|_0)=(E_n,\|\cdot\|_n)$ and we denote by $B_k$ the unit ball of $(E_k,\|\cdot\|_k)$ and by $0_k$ the zero of $E_k$ for $k=0,...,n$. We fix some isometry between $(E_0,\| \cdot \|_0)$ and $(\mathbb{R}^2,\| \cdot \|)$ with the usual Euclidean norm $\| \cdot \|$ and we again denote by $\mathfrak{C}_\alpha$ the corresponding (via the isometry) family of cones in $E_0$.

 Let $\mathcal{F}:=(f_k:B_k\rightarrow E_{k+1})_{0\leq k< n}$ be a sequence of $C^1$ maps such that $f_k(0_k)=0_{k+1}$ and $f_k$ is a diffeomorphism onto its image  for all $k$. For $0<l\leq n$ we denote by  $f^l=f_{l-1}\circ...\circ f_0$  the $l^{th}$-composition defined on the $l$-dynamical ball $B(\mathcal{F},l):=\{x\in B_0, \ f^jx\in B_j \text{ for }0\leq j<l-1\}$. By convention we let $f^0$ to be the identity map on $B_0$.

\begin{defi}\label{dsd}
  For $n\in \mathbb{N}$, $C>0$,  $\delta>0$ and $\alpha>0$, an open subset $U_n$ of $B(\mathcal{F},n)\subset E_0$ is said to be $(C,\delta,\alpha,n)$-hyperbolic (or shortly $n$-hyperbolic) when there are two bicones  $ (C_u,C_s)$ in $\mathfrak{C}_\alpha$ and  two (transverse) $C^\infty$ smooth unit vector fields $e_u:U_n\rightarrow C_u$ and $e_s:U_n\rightarrow C_s$ such that for any $y\in U_n$  we have :
\begin{itemize}
 \item $T_yf^n \left(e_{u}(y)\right)\in C_u$ with $\|T_yf^n\left(e_u(y)\right)\|_n\geq Ce^{n\delta}$,
 \item $T_yf^n \left(e_{s}(y)\right)\in C_s$  with $\|T_yf^n\left(e_s(y)\right)\|_n\leq C^{-1}e^{-n\delta}$.
 \end{itemize}
  The vector fields $e_u$ and $e_s$ are called respectively the $n$-expanding and $n$-contracting fields of the $n$-hyperbolic set $U_n$.
\end{defi}


If we let $v_u$ be an oriented center of the cone $C_u$, then by changing $e_u$ for  $-e_u$ we may assume $\angle \left(v_u,e_{u}(y)\right)\leq \frac{\pi}{2}$ for all $y\in U_n$. 
 As the cones $C_u$ and $C_s$  are transverse  and have the  same aperture
 we can write $v_u=\lambda  e_s +\beta e_u$ with  $|\lambda| \leq |\beta|$. Therefore, for $C$, $\alpha$  and $\delta$ fixed, provided  $Ce^{n\delta}$ is large enough, the vector $T_xf^n(v_u)$ belongs to $C_u(3/2)$ and $\|T_xf^n(v_u)\|\geq \frac{C}{3}e^{n\delta}$. Thus we may always assume $e_u=v_u$ for large $n$ if we  relax the first item of Definition \ref{dsd} by replacing $C_u$ by $C_u(3/2)$ and $C$ by $C/3$. 

 \begin{rem}
The $n$-expanding and $n$-contracting fields are not canonical.  When working later on with the dynamical system given by a surface diffeomorphims (and not a sequence $\mathcal{F}$ as above), they will not a priori correspond with the Oseledets unstable and stable directions, which do not vary smoothly and are not globally defined.  Indeed the $n$-expanding and $n$-contracting fields only depend on the $n$-first iterations of the dynamical system. However for a saddle hyperbolic $n$-periodic point $y$ lying in a $n$-hyperbolic set $U_n$ the vectors  $e_{s}(y)$ and $T_yf^ne_u(y)$ associated to $U_n$ are respectively close to the usual  (Oseledets) stable and unstable spaces (whenever $Ce^{n\delta}$ is large enough).
\end{rem}

\subsection{Hyperbolic hexagons}\label{ssecc}
We define in this subsection hyperbolic hexagons as a generalization of the usual notion of rectangles that we first recall now.
Let $U$ be an open set of $E_0$. For two transverse one dimensional foliations $\mathcal{F}_u$ and $\mathcal{F}_s$ on $U$ a  rectangle is the image of a bifoliation chart, i.e of a topological embedding  $\phi:[0,1]^2\rightarrow U$ which straightens simultaneously the stable and unstable foliations: for all $x$ and $y$ in $[0,1]$ the sets $\phi(\{x\}\times [0,1])$ and $\phi([0,1]\times  \{y\})$ are pieces of a  leaf of  $\mathcal{F}_u$ and $\mathcal{F}_s$ respectively.

 If $\mathcal{F}_u$ and $\mathcal{F}_s$ are generated by two transverse $C^\infty$ smooth non vanishing vector fields $e_u$ and $e_s$ on $U$  we introduce the following generalization of rectangles :
  \begin{defi}\label{dref} With the previous notations an open  subset $H$ of $U$ is said to be an hexagon  when any two points in $H$ may be joined by a $C^1$ regular curve  in $H$ which is either tangent everywhere  to $C(e_u,e_s)$ or  tangent everywhere  to  $C(e_u,-e_s)$.
 \end{defi}

 For $\epsilon\in \{-1,1\}$ we let $\mathcal{F}_u^\epsilon$ and $\mathcal{F}_s^{\epsilon}$ be the oriented one dimensional foliations associated to $\epsilon e_u$ and $\epsilon e_s$. Then an hexagon is an open set satisfying the following (oriented) accessibility-like property \footnote{See \cite{BP} for the usual notion of accessibility in smooth dynamical systems.}:

 \begin{lemma}
An open  subset $H$ of $U$ is an hexagon if and only if for any $x,y\in H$ there is $\epsilon_u,\epsilon_s\in \{-1,1\}$ and $x=z_1,...,z_p=y\in U$ such that for all $1\leq i\leq p$ the points $z_{i-1}$ and $z_i$ are in the same leaf of $\mathcal{F}_\alpha$ and the path in this leaf  from $z_{i-1}$ to $z_i$ is oriented as $\mathcal{F}_\alpha^{\epsilon_\alpha}$  for  $\alpha=u$ or $s$.
 \end{lemma}

In the above statement we  point out that $p$ may depend on $x$ and $y$, but $\epsilon_u, \epsilon_s$ are independent of $i$.  

  When an hexagon $H$ is a subset of a rectangle given by a bifoliation $C^\infty$ smooth chart $\phi:[0,1]^2\rightarrow U$ (it will be always the case in the following) then $\phi^{-1}(H)$ is an hexagon in $]0,1[^2$  for the usual  foliations in vertical and horizontal lines. In this case we have:
\begin{lemma}\label{toutou}
Let $H$ be an open subset of $]0,1[^2$ then the following conditions are equivalent :
\begin{itemize}
\item $H$ is an hexagon for the horizontal and vertical foliations,
\item there are $0<a<b<1$ and $c,d\in [a,b] $ as well as functions $\zeta<\eta:]a,b[\rightarrow ]0,1[$  such that:
\begin{itemize}
\item $\eta$ (resp. $\zeta$) is lower semi-continuous (resp. upper semi-continuous),
\item $\eta$ (resp. $\zeta$) is nondecreasing on $]a,c[$ (resp. $]d,b[$) and nonincreasing on $]c,b[$ (resp. $]a,d[$),
\end{itemize}
 $$ \text{ with} \footnote{Such sets look like usual hexagons. This explains the terminology.} \ H=\left\{ (x,y) \in ]0,1[^2,\ \zeta(x) <  y< \eta(x)\right\}.$$
\item any pair of points in $H$ may be joined by a monotone staircase function (with finitely many steps).
\end{itemize}
\end{lemma}

The easy proofs of the two above lemmas are left  to the reader. We will now consider hexagons with respect to the contracting and expanding fields associated to a $n$-hyperbolic structure.

\begin{defi}
Let $U_n$ be a  $n$-hyperbolic set of $\mathcal{F}$  as in the Subsection 5.2. Let $e_u$ and $e_s$ be the $n$-expanding and $n$-contracting associated vector fields.

A subset $R_n$ of  $U_n$  is called a  $n$-hyperbolic hexagon when $R_n$ is an hexagon  with respect to $e_u$ and $e_s$.  When $U_n$ is an $(C,\delta,\alpha,n)$-hyperbolic set with respect to the constants  we will say that $R_n$ is a $(C,\delta,\alpha,n)$-hyperbolic hexagon. 

A set $\tilde{R}_n$ is called a generalized $n$-hyperbolic hexagon when there are $C,\delta,\alpha$ such that $\tilde{R}_n$ is the Hausdorff limit of the closures of a sequence of $(C,\delta,\alpha,n)$-hyperbolic hexagons (then we also say that $\tilde{R}_n$ is  a generalized  $(C,\delta,\alpha,n)$-hyperbolic hexagon). 
\end{defi}


We say $x\in B(\mathcal{F},n)$ is a $n$-periodic point for  $\mathcal{F}:=(f_k:B_k\rightarrow E_{k+1})_{0\leq k< n}$ when $f^nx=x$. The following key lemma will allow us to bound the number of periodic points by counting the number of hyperbolic
hexagons covering a dynamical ball.

\begin{lemma}\label{ff}
For all  $C$, $\delta$, $\alpha$,  there is an integer $N=N(C,\delta,\alpha)$ such that any generalized $(C,\delta,\alpha,n)$-hyperbolic hexagon carries with $n>N$ at most one $n$-periodic point for $\mathcal{F}$.
\end{lemma}

\begin{proof}
Let us first consider  a $(C,\delta,\alpha,n)$-hyperbolic hexagon $R_n$  with respect to $(C_u,C_s)\in \mathfrak{C}_\alpha$. Assume by contradiction that $R_n$ contains two $n$-periodic points $p$ and $q$. We can assume (by exchanging $e_u$ with $-e_u$ or/and $e_s$ with $-e_s$) that there is  a $C^1$ regular curve $\lambda:[0,1]\rightarrow M$ with $\lambda(0)=p$ and $\lambda(1)=q$  such that $\lambda'(t)\in C(e_s,e_u)$ for all $t\in [0,1]$.  Thus one can write   $\lambda'(t)$ as $\lambda'_u(t)+\lambda'_s(t)$ with $\lambda'_u(t)\in \mathbb{R}^+e_u$ and $\lambda'_s(t)\in \mathbb{R}^+e_s$.  Also for $\gamma=u,s$ we let $x_{\gamma}:=\int_{[0,1]} \lambda'_{\gamma}(t) dt\in C_{\gamma}$  and  $w_{\gamma}:=\int_{[0,1]} T_{\lambda(t)} f^n\left(\lambda'_{\gamma}(t)\right) dt\in C_{\gamma}$. For some positive constant $C(\alpha)$ depending only on $\alpha$ we have (we write $\| \cdot \|$ for $\| \cdot \|_0=\| \cdot \|_n$ to simplify the notations):   \begin{eqnarray*}
\|w_u\|&\geq & C(\alpha)\int_{[0,1]} \|T_{\lambda(t)} f^n\left(\lambda'_{u}(t)\right)\| dt,\\
& \geq &  CC(\alpha)e^{n\delta}\int_{[0,1]}  \|\lambda'_{u}(t)\| dt,\\
&\geq & CC(\alpha)e^{n\delta}\|x_u\|,
\end{eqnarray*}
and we get similarly
$$\|w_s\|\leq C^{-1}C^{-1}(\alpha)e^{-n\delta}\|x_s\|.$$ Thus for $n$ large enough  $w_u-x_u\in C_u(2)$ and $w_s-x_s\in C_s(2)$. This contradicts the fact that $p-q=x_u+x_s=w_u+w_s$ and the transversality of the cones $C_u(2), C_s(2)$.
Finally if $\tilde p$ and $\tilde q$ are $n$-periodic points of a generalized $n$-hyperbolic hexagon $\tilde {R}_n$, which is the Hausdorff limit of $\left(\overline{R_n^k}\right)_k$ where $(R_n^k)_k$ is a sequence of $(C,\delta,\alpha,n)$-hyperbolic hexagons. Then $p-q=x_u+x_s\simeq w_u+w_s$
 for any $p,q$ in $R_n^k$ (with some large $k$) respectively  close to $\tilde{p}$ and $\tilde q$. It leads  to the same contradiction. This concludes the proof of the lemma.
\end{proof}

\begin{rem}\label{touto}
 In the above proof  we used in a essential way  that  the  $n$-expanding and $n$-contradicting fields, $e_u$ and $e_s$,  and their image under $Tf^n$ lie  in the same bicone $(C_u,C_s)$.  It is not enough to assume their image lies  in another bicone of $\mathfrak{C}_\alpha$.
\end{rem}

\section{Semi-algebraic tools}
We now recall some results of semi-algebraic geometry which will be used in the next section to build a collection of $n$-hyperbolic hexagons covering a given $n$-dynamical ball. We refer to \cite{BR} for an introduction to semi-algebraic geometry.\\

\subsection{Nash maps and their degree}

A semi-algebraic set of $\mathbb{R}^d$ is a set which may be written as a finite union of polynomial inequalities.
Semi-algebraic sets are finite union of real analytic manifolds (also called Nash manifolds). A map $f:A\subset \mathbb{R}^d\rightarrow \mathbb{R}^e$ is called semi-algebraic when its graph $\Gamma_f:=\{(x,f(x)), \ x\in A\}\subset \mathbb{R}^{d+e}$ is semi-algebraic. A real analytic semi-algebraic map  on a Nash manifold is called a Nash map. Here we only consider Nash maps defined on cubes $]0,1[^d$ for $d\in \mathbb{N}$ (by convention we let $]0,1[^0=\{0\}$). \\

The algebraic complexity of a Nash map may be quantified as follows.  The complexity $\com(f)$ of a   Nash map $f:]0,1[^d\rightarrow \mathbb{R}^e$  is the minimal integer $n$ for which we have 
$$\Gamma_f=\bigcup_{i=1,...,n}\bigcap_{j=1,...,n}\{P_{i,k} \ ?_{i,j} \ 0 \}$$
 for  $?_{i,j}\in\{>,=\}$ and for polynomials $P_{i,j}\in \mathbb{R}[X_1,...,X_{d+e}]$ with total degree\footnote{The total degree of a monomial $\prod_l X_l^{\beta_l}$ is the sum $\sum_l\beta_l$ and the total degree of a polynomial is defined to be the largest total degree of its monomials.} less than or equal to $n$.  
 
 In the present paper we will work with another notion. The degree $\deg(f)$ of a real Nash map $f:]0,1[^d\rightarrow \mathbb{R}$ is the minimal total degree of the vanishing polynomials of $f$, i.e. polynomials $P\in \mathbb{R}[X_1,...,X_d,X_{d+1}]\setminus \{0\}$ with $P(x,f(x))=0$ for any $x\in ]0,1[^d$. The map $f$ being real analytic  we may assume the polynomial $P$ to be irreducible.  When $f$ is a polynomial map defined by a polynomial   $F\in \mathbb{R}[X_1,...,X_{d}]$ the degree $\deg(f)$ of $f$ is then equal to the (usual) total degree $\deg_t(F)$ of $F$. For a Nash map $\phi=(\phi_1,...,\phi_e):]0,1[^d\rightarrow \mathbb{R}^e$ we let the degree of $\phi$ be the maximal  degree of its components $\phi_1,...,\phi_e$.
 
 In general the graph of a real Nash map may not be written as the zero locus of a vanishing polynomials. However , the degree and the complexity of a Nash map are equivalent in the following sense (Proposition 4 in \cite{cos}) : there is a function $a=a_{d,e}:\mathbb{N}\rightarrow \mathbb{N}$ such that for any Nash map  $f:]0,1[^d\rightarrow \mathbb{R}^e$ we have 
 \begin{eqnarray}\label{equi}
 \deg(f)\leq a(\com(f)) \text{ and } \com(f)\leq a(\deg(f)).\\ \nonumber
 \end{eqnarray}
 


\subsection{Yomdin-Gromov Lemma}\label{yy}
A fundamental tool in the Reparametrization lemma of dynamical balls by   hyperbolic hexagons  is the following powerful lemma due to Yomdin and Gromov \cite{Yom}\cite{Gr}. We recall the functional version of this lemma in dimension 3 (we later apply this lemma in local charts of the unit tangent bundle of a compact surface). 

\begin{defi}A  map $\phi=(\phi_1,...,\phi_d) :]0,1[^d \rightarrow \mathbb{R}^d$ is said triangular when the $i^{th}$-component $\phi_i$ of $\phi$ only depends on the $i^{th}$-first coordinates, i.e. for all $(x_1,...,x_d)\in]0,1[^d$ we have$$\phi(x_1,...,x_d)=(\phi_1(x_1),\phi_2(x_1,x_2),...,\phi_d(x_1,...,x_d)).$$
\end{defi} Observe that triangular  maps are stable under  composition (when the composition  is well defined). \\

When $r$ is an integer the $C^r$ norm $\|\psi\|_r$ of a $C^r$ map $\psi:U\rightarrow \mathbb{R}^e$ for an open set $U\subset \mathbb{R}^d$ is defined as follows:
$$\|\psi\|_r:=\max_{k=0,1...,r}\|\psi^{(k)}\|_\infty$$
$$ \text{ with }  \|\psi^{(k)}\|_\infty=\max_{\alpha\in \mathbb{N}^d, \ |\alpha|=k}\sup_{x\in U}\|\partial^\alpha \psi (x)\| \text{ for } k\in \mathbb{N}.$$

\begin{theo}\label{firs}\cite{Gr} (see also \cite{BI}\cite{PI})
Let $r$ be a positive integer and let $f:]0,1[^3\rightarrow \mathbb{R}^4$ be a  Nash map. Then there is a family  $\mathcal{F}$ of Nash maps $\phi=(\phi_1,\phi_2,\phi_3):]0,1[^{d_\phi}\rightarrow ]0,1[^3$ with $0\leq d_\phi\leq 3$ such that:

\begin{itemize}
\item each $\phi\in \mathcal{F}$ with $d_\phi=3$  is  triangular and is a diffeomorphism onto its image,
\item $f^{-1}(]-1,1[^4)\subset \bigcup_{\phi\in \mathcal{F}}\phi(]0,1[^{d_\phi})$,
\item $\|\phi\|_r,\|f\circ \phi\|_r\leq 1$ for any $\phi\in \mathcal{F}$,
\item $\sharp \mathcal{F}$ and $\max_{\phi \in \mathcal{F}}\deg(\phi)$ is bounded from above by a constant which depends only on $\deg(f)$ and $r$.
\end{itemize}
\end{theo}

The key point in the above statement lies in the last property : the number of reparametrizations used to control the derivatives of $f$ depends only on the degree of $f$, not on its $C^r$ norm $\| \cdot \|_r$. This lemma was introduced by Yomdin and Gromov \cite{Yom} \cite{Gr} to bound the local dynamical complexity of $C^r$ maps through Taylor-Lagrange polynomial interpolation. More recently
 applications in diophantine geometry were discovered by Pila and Wilkie (see the survey \cite{yomm}).

\begin{rem}\label{pr}
The reparametrization maps $\phi$ are  uniformly continuous (in fact $1-$Lipschitz) and therefore extend continuously on the closed cube $[0,1]^{d_\phi}$. Thus we may assume $d_\phi=3$ for all $\phi$ in the above statement if one replaces
the second item by $$f^{-1}(]-1,1[^4)\subset \bigcup_{\phi\in \mathcal{F}}\phi([0,1]^3).$$
\end{rem}

\begin{rem}
In \cite{Gr} (also in \cite{Bl}) the authors have worked with  the complexity and not the degree, but this is irrelevant as these notions are equivalent by  (\ref{equi}).
\end{rem}


\subsection{Degree of a composition}\label{62}
For the purpose of the next section we need to control dynamically the degree of semi-algebraic maps. Using elimination the following rule of composition for the degree was proved in \cite{b}.

\begin{lemma}\label{res}
Let $\phi_0:]0,1[^e\rightarrow ]0,1[$ and
$\phi_1,...,\phi_e:]0,1[^d\rightarrow ]0,1[$ be Nash maps then
$$\deg\left(\phi_0(\phi_1,\phi_2,...,\phi_e)\right)\leq \prod_{i=0,1,2,...,e} \deg(\phi_i).$$
\end{lemma}

The lemma is obviously satisfied for polynomial maps $(\phi_i)_i$. In fact we have in this case
$\deg\left(\phi_0(\phi_1,\phi_2,...,\phi_e)\right)\leq  \deg(\phi_0)\max_{i=1,2,...,e} \deg(\phi_i)$.\\

We recall some basic facts of elimination theory used in the proof of Lemma \ref{res}. Let $A$ be a factorial ring and $P,Q\in A[X]$.
The resultant $\Res_X(P,Q)\in A$, which  is the determinant of the Sylvester matrix of $P$ and $Q$, satisfies the following properties (see \cite{Lang}):
\begin{itemize}
\item $\Res_X(P,Q)=0$ if and only if $P$ and $Q$ have a non constant common factor in $A[X]$,
\item  $\Res_X(P,Q)$ belongs to the ideal $\langle P,Q\rangle$ generated by $P$ and $Q$.
\end{itemize}

When $A=\mathbb{R}[X_1,...,X_n]$ the resultant $\Res_X(P,Q)$ is a polynomial in $X_1,...,X_n$ and the total degree of $\Res_X(P,Q)\in \mathbb{R}[X_1,...,X_n]$ is bounded from above by the product of the total degrees of $P$ and $Q$,  seen as polynomials in the $n+1$ variables $X_1,...,X_n,X$ (Theorem 10.9 in \cite{thes}). In this  case we note also that  any common root $(x_1,...,x_n,x)\in \mathbb{R}^{n+1}$ of $P$ and $Q$ satisfies,
$$\Res_{X}(P,Q)(x_1,...,x_n)=0.$$

We go back now to the proof of Lemma \ref{res}.
\begin{proof}
For $i=0,1,2,..., e$ we let $P_i$ be a vanishing irreducible polynomial of $\psi_i$ with minimal total degree.
We eliminate the variable $Y_1$ in

$$\left\{\begin{array}{rl}P_0(Y_1,...,Y_e,Y)&=0,\\
P_1(X_1,...,X_d,Y_1)&=0.
\end{array}\right.$$

Observe that for all $(x_1,,...,x_d,y_2,...,y_e)\in [0,1]^{d-1+e}$ we have
$$\left\{\begin{array}{rl}P_0\left((\phi_1(x_1,...,x_d),y_2,...,y_e,\phi_0\left(\phi_1(x_1,...,x_d),y_2,...,y_e\right)\right)&=0,\\
P_1\left(x_1,...,x_d,\phi_1(x_1,...,x_d)\right)&=0,
\end{array}\right.$$
and therefore the polynomials $P_0$ and $P_1$ as elements of $\mathbb{R}[X_1,...,X_d,Y_1,...,Y_e,Y]$ have  the common root
$$\left(x_1,...,x_d,\phi_1(x_1,...,x_d),y_2,...,y_e, \phi_0(\phi_1(x_1,...,x_d),y_2,...,y_e\right).$$ In particular the resultant in $Y_1$ of $P_0$ and $P_1$, say $Q_1:=\Res_{Y_1}(P_0,P_1)$ in\\ $\mathbb{R}[X_1,...,X_d,Y_2,...,Y_e,Y]$,  vanishes at $\left(x_1,...,x_d,y_2,...,y_e, \phi_0\left(\phi_1(x_1,...,x_d),y_2,...,y_e\right)\right)$ for all $(x_1,...,x_d,y_2,...,e_d)\in ]0,1[^{d-1+e}$. Moreover we have $\deg_t(Q_1)\leq \deg_t(P_0)\deg_t(P_1)$. This proves the lemma for $e=1$ because in this case $Q_1$ is a vanishing polynomial of $\phi_0\circ \phi_1$.

For larger integers $e$ we eliminate by induction on $1\leq k< e$  the variable $Y_{k+1}$ in $$\left\{\begin{array}{rl}Q_{k}(X_1,...,X_d,Y_{k+1},...,Y_e,Y)&=0,\\
P_k(X_1,...,X_d,Y_{k+1})&=0.
\end{array}\right.$$
with $Q_{k}=\Res_{Y_{k}}(P_k,Q_{k-1})\in \mathbb{R}[X_1,...,X_d,Y_{k+1},...,Y_e,Y]$. As $(P_i)_i$ are irreducible,
the polynomials $Q_{k}$ are not zero. Also $Q_e=\Res_{Y_{e}}(P_e,Q_{e-1})$ is a vanishing polynomial of $\psi_0\left(\psi_1,\psi_2,...,\psi_d\right)$.
Finally we have $\deg_t(Q_e)\leq \prod_{i=0,1,..,e} \deg_t(P_i)= \prod_{i=0,1,...,e} \deg(\psi_i)$.
\end{proof}

Then we get by an immediate induction that the algebraic degree grows at most exponentially in
$n$ after $n$-compositions. This estimate on the dynamical degree will be crucial to control the number of hexagons covering a given dynamical ball (Lemma \ref{tec}).
\begin{cor}\label{compo}
Let $\psi_1,\psi_2,...,\psi_n:]0,1[^d\rightarrow ]0,1[^d$ be Nash maps then
$$\deg(\psi_n\circ...\circ \psi_2\circ \psi_1)\leq \deg(\psi_n\circ...\circ \psi_2) \deg(\psi_1)^d  \leq \deg(\psi_n)\prod_{i=1,...,n-1}\deg(\psi_i)^d.$$
\end{cor}

This result first appears in \cite{b} to  compute the dimensional entropies of a product.

\subsection{Covering  semialgebraic hyperbolic sets by  hexagons}
We cover now an $n$-hyperbolic semi-algebraic set  by a collection of $n$-hyperbolic hexagons    with cardinality bounded by a function depending only on the algebraic complexity of this set. Let us specify the precise form of the semi-algebraic sets that we will consider.

Fix a pair $(C_u,C_s)$ of  cones in $E_0$ belonging to $ \mathfrak{C}_\alpha$. Then we may choose an isometry between $(E_0,\| \cdot \|_0)$ and $(\mathbb{R}^2,\| \cdot \|)$ so that the center 
 of $C_u$ is mapped to $\{0\} \times \mathbb{R}$.  Using the notations of Subsection \ref{ffh} we consider a $n$-hyperbolic set $U_n$ with respect to $\mathcal{F}$ and $(C_u,C_s)$ of the following form  : there are Nash maps $\phi : ]0,1[^2\rightarrow B$ and $\psi:]0,1[^3\rightarrow \mathbb{T}$ (with $B$ and $\mathbb{T}$ being respectively the Euclidean unit ball and unit sphere of $\mathbb{R}^2$) such that:

\begin{itemize}
\item $\phi$ is a diffeomorphism onto its image and with $\underline{U_n}$ denoting  the isometric image of $U_n$ in $\mathbb{R}^2$ we have 
 $$\underline{U_n}=\phi(]0,1[^2),$$
\item the $n$-expanding vector field $\underline{e_u}$  on $\underline{U_n}$ is given by 
$$\underline{e_u}( \phi(t,s)) = \underline{v_u}=(0,1) \text{ for all }(t,s)\in ]0,1[^2,$$
\item the $n$-contracting vector field $\underline{e_s}$ on $\underline{U_n}$ is given by 
$$\underline{e_s}( \phi(t,s)) =  \psi(t,s,0) \text{ for all }(t,s)\in ]0,1[^2.$$
\end{itemize}

We recall the  $n$-expanding and $n$-contradicting  vector fields given by the $n$-hyperbolic structure on $U_n$ (cf. Definition \ref{dsd} ) do not correspond to the Oseledets distribution.  

\begin{lemma}\label{d}
Let $(U_n,\phi,\psi)$ be as above. Then  there is  a collection $\mathcal{H}$ of  generalized $n$-hyperbolic hexagons such that 
$$U_n\subset \bigcup_{H_n\in \mathcal{H}}H_n\text{ and }$$

$$\sharp \mathcal{H}\leq P\left(\deg(\phi), \deg(\psi)\right),$$ 
for a universal polynomial $P$.
\end{lemma}

\begin{proof}
We have $U_n=\phi(]0,1[^2)=\bigcup_{0<a<1/2}\phi(I_a)$ with $I_a=]a,1-a[^2$. Thus by compactness it is enough  to cover for any $0<a<1/2$ the set   $\phi(I_a)$ by  generalized $n$-hyperbolic hexagons as in  the lemma with a universal polynomial $P$ which does not depend on $a$. 

We fix once and for all  $0<a<1/2$. As $\phi:]0,1[^2\rightarrow \mathbb{R}^2$ is a diffeomorphism onto its image we 
may extend smoothly on the whole space $\mathbb{R}^2$ the restriction of $\underline{e_s}$ on $\phi(I_a)$ in 
such a way the extended $C^\infty$ field, denoted by  $\underline{e_s^a}$, does not vanish and  belongs also to the cone  $C_s$. The vector 
fields $\underline{e_s^a}$ and $\underline{e_u}=(0,1)$ being smooth non vanishing  vector  fields  on $\mathbb{R}^2$ lying in the transverse cones $C_s$ and $C_u$  there is a diffeomorphism $\Theta:[0,1]^2\rightarrow \mathbb{R}^2\supset \phi(]a,1-a[^2)$ 
which straightens simultaneously the foliations associated to $\underline{e_s^a}$ and $\underline{e_u}$, i.e. $
\Theta(\{t\}\times [0,1])$ and  $\Theta([0,1]\times \{s\} )$ are respectively  tangent to $\underline{e_u}$ and $\underline{e_s^a}$  for all $(t,s)\in [0,1]^2$. 

We let $\partial I_a$ (resp. $\ex(I_a)$) be  the boundary of $I_a$ given by the four edges of the corresponding square (resp. the extreme set of $I_a$ given by the four corners). A point $x$ in $\partial\phi(I_a)=\phi(\partial I_a)$ will be said \textit{singular} when :
\begin{itemize}
\item either $x\in \phi(\ex(I_a))$, 
\item or  $\underline{e_s}$ is tangent to $\partial\phi(I_a)$ at $x$, but $\underline{e_s}$ is transverse \footnote{As $\phi$ and $e_s$ are real analytic it is equivalent to say the image by $\phi$ of the edge of $I_a$ containing $\phi^{-1}(x)$ is not an integral curve of the vector field $e_s$.} to $\partial\phi(I_a)$  on $U\setminus \{x\}$ for a neighborhood $U$ of $x$,
\item or  $\underline{e_u}$ is tangent to $\partial\phi(I_a)$ at $x$, but $\underline{e_u}$ is transverse to $\partial\phi(I_a)$  on $U\setminus \{x\}$ for a neighborhood $U$ of $x$.
\end{itemize}

\begin{CClaim}\label{on}Let  $\mathcal{S}$ be the set of singular points. Then we have $$\sharp \mathcal{S}\leq P\left(\deg(\phi), \deg(\psi)\right),$$ for a universal polynomial $P$ (which does not depend on $a$)
\end{CClaim}

We consider the subset $\mathcal{S}_x$ of $[0,1]$ given by the $x$-coordinates of the singular points through the chart $\Theta$, i.e.
$$\mathcal{S}_x=\pi(\Theta^{-1}\mathcal{S}),$$
where $\pi:[0,1]^2\rightarrow [0,1]$ is the projection on the first coordinate. Let us reorder this set  as follows $$\mathcal{S}_x:=\{x_1<x_2<...<x_K\}.$$ 

\begin{CClaim}\label{tw}
For any $1\leq l <K$ there are monotone  maps $\eta_1^l< ... <\eta_{N_l}^l:]x_l,x_{l+1}[\rightarrow [0,1]$ with $ N_l\leq \sharp \mathcal{S}$ such that 
\begin{eqnarray*}\left(]x_l,x_{l+1}[\times [0,1]\right)\cap \Theta^{-1}\circ\phi(I_a)&:=\bigcup_{1\leq i< N_l}\{(x,y), & \ x\in ]x_l,x_{l+1}[, \\
& & \eta_1^{i+1}(x) < y< \eta_1^{i+1}(x)\}.
\end{eqnarray*}

\end{CClaim}

As already observed in Lemma \ref{toutou} the sets 
$E_l:=\{(x,y),  \ x\in ]x_l,x_{l+1}[, \ \eta_1^{i+1}(x) < y< \eta_1^{i+1}(x)\}$ are  hexagons for the horizontal and vertical foliations, and thus 
the sets $\Theta(E_l)$ are $n$-hyperbolic hexagons. Therefore Lemma \ref{d} 
will be proved once we have shown Claim \ref{on} and Claim \ref{tw}.

\begin{proof}[Proof of Claim \ref{on}]
For a singular point $x\notin \phi(\ex(I_a))$  either $\underline{e_u}$ or $\underline{e_s}$ is tangent to $\partial \phi(I_a)$ at $\Theta(x)$. 
Thus it is enough to prove that  the wedge products $(0,1)\wedge \partial_t\tilde \phi(t,i)$ and $\psi(t,i,0)\wedge\partial_t\tilde \phi(t,i)$  (resp. $(0,1)\wedge \partial_s\tilde \phi(i,s)$ and $\psi(i,s,0)\wedge\partial_s\tilde \phi(i,s)$) for $i\in  \{a,1-a\}$  vanishes either for all $t\in ]0,1[$ (resp. for all $s$) or for $N$ values of $t$ (resp. $s$) with $N\leq P(\deg(\phi), \deg(\psi))$ where $P$ is a universal polynomial (independent of $a$). But these maps  being  Nash on $]0,1[$ either they vanish everywhere or the number of their zeroes is less than or equal to their degree.  By using again elimination theory one easily checks that for any Nash maps $f,g: ]0,1[\rightarrow \mathbb{R}$ the degrees of $fg$, $f+g$ and $f'$ depends polynomially on the degree of $f$ and $g$, i.e. there are universal polynomials $Q\in \mathbb{R}[X,Y]$ and $R\in \mathbb{R}[X]$ with $\deg(fg), \deg(f+g)\leq Q(\deg(f),\deg(g))$ and $\deg(f')\leq R(\deg(f))$. Moreover for a Nash maps $h:]0,1[^2\rightarrow \mathbb{R}$ and for any $c\in ]0,1[$ the degree of $h_c:=h(c,.)$ is less than or equal to the degree of $h$. As the above wedge products are obtained by such operations on $\psi$ and the coordinates of $\phi$,  this concludes the proof of the lemma. 
\end{proof}

\begin{proof}[Proof of Claim \ref{tw}]
 \begin{figure}[!h]
\vspace{-0,5cm}
\includegraphics[scale=0.7]{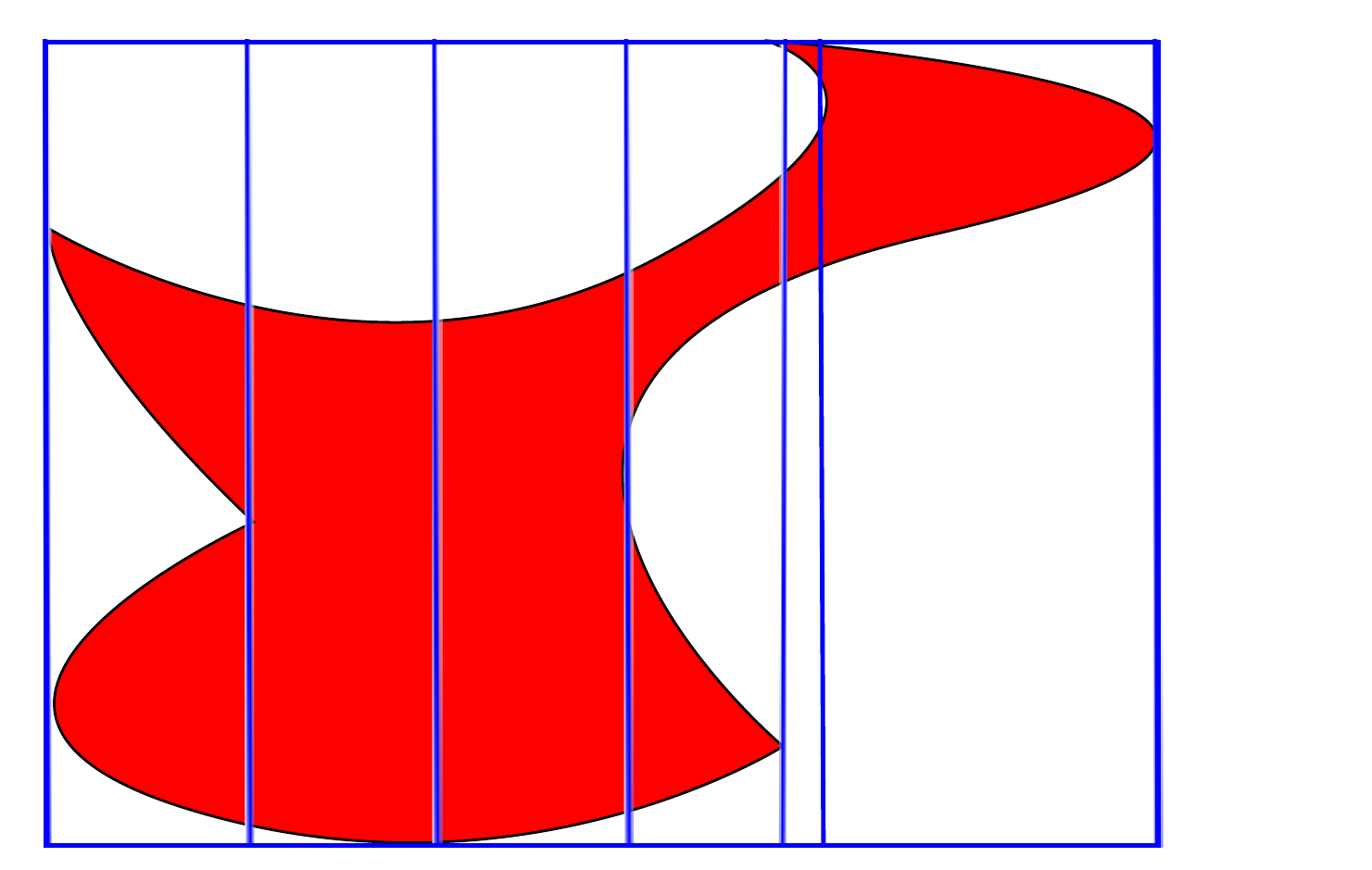}
\centering
\caption{\label{ffi}\textbf{The  $n$-hyperbolic  set $\phi(I_a)$ (in red) through a straightening chart of the $n$-contracting and $n$-expanding foliations.} The blue vertical lines represent the expanding leaves in the domain of the  chart at the singular points of  $\partial \phi(I_a)$.}
\end{figure}

It is easily checked from the definition of singular points that the set $\left(]x_l,x_{l+1}[\times [0,1]\right)\cap \phi(I_a)$ may be written as the union of "slices" of the form $\{(x,y),  \ x\in ]x_l,x_{l+1}[, \ \eta(x) < y< \zeta(x)\}$ with $\eta$ and $\zeta$ being monotone (cf. Figure \ref{ffi}). Finally let us  see why  the number of such slices is less than or equal to $\sharp \mathcal{S}$. The boundary of $\Theta^{-1}\circ\phi(I_a)$ is a Jordan curve $\gamma:[0,1]\rightarrow [0,1]^2$  given by four $C^\infty$ smooth Jordan arcs (corresponding to the image of the edges of $I_a$).  Moreover the graph of any $\eta_i^l$ (with $\eta_i^l$ as in the statement of  Claim
\ref{tw})  is a piece of $\gamma$, i.e. there is $0\leq a_{i,l}<b_{i,l}\leq 1$ such that  the graph of $\eta_i^l$ coincides with $\gamma(]a_{i,l},b_{i,l}[)$.   Finally notice that 
$\mathcal{S}\subset \bigcup_{i,j}\{\gamma(a_{i,l}),\gamma(b_{i,l}) \}$. Let  $i\neq  j$. We may assume  $a_{j,l}>b_{i,l}$ and $\pi(\gamma(b_{i,l}))=x_l$ (the other cases being similar). Then either $\Theta^{-1}\circ\phi(\ex(I_a))\cap \gamma([b_{i,l},a_{j,l}])\neq \emptyset$ or $\gamma(a_{j,l})$ and $\gamma(b_{i,l})$ belongs to the same $C^\infty$ smooth (open) arc. But as $(x_{l+1}=)\pi(\gamma(a_{i,l}))>\pi(\gamma(b_{i,l}))=x_l\leq \pi(\gamma(a_{j,l}))$, there is in this case $t\in [b_{i,l},a_{j,l}] $ such that the vertical line  is tangent to  $\gamma$ at $\gamma(t)$, i.e. $\underline{e_u}$ is tangent to $\partial \phi(I_a)$ at $\Theta\circ \gamma(t)$. In both  cases there is  $t\in [b_{i,l},a_{j,l}]$ such that $\Theta \circ \gamma(t)$ belongs to $\mathcal{S}$. In particular the number of function $\eta_i^l$ above $]x_l,x_{l+1}[$ is less than or equal to $\sharp \mathcal{S}$.

\end{proof} 
\end{proof}

\section{Reparametrization Lemma}

We prove now a Reparametrization  Lemma of dynamical balls for surface diffeomorphisms as Lemma \ref{main} for interval maps. We follow a general scheme in Yomdin's theory by first approximating locally our given $C^r$ diffeomorphism by polynomials and by then applying semi-algebraic tools. The main novelty in the present paper consists in approximating not only the $C^r$ map but also  the action of its derivative on the unit tangent bundle. A similar approach was developed in \cite{Burguet} to build  symbolic extensions  of $C^r$ smooth surface systems with $r>1$.

\subsection{Local dynamic}
Let $f:M\rightarrow M$ be a $C^r$  diffeomorphism on a Riemanian surface $(M,\| \cdot \|)$ with $r>1$. We fix a point $x\in M$, a scale $\epsilon$ less than the radius of injectivity 
of $(M,\| \cdot \|)$ and two positive integers $p<<n$. Let $\Exp$ be the exponential map on the tangent bundle $TM$. For $0\leq k\leq [n/p]$ (resp. $k=[n/p]+1$) we  let $(E_k,\| \cdot \|_k)$  be the Euclidean space $(T_{f^{kp}x}M, \| \cdot \|_{f^{kp}x})$ (resp.  $(T_{f^{n}x}M, \| \cdot \|_{f^{n}x})$)  and  $B_k$ be its unit ball.  We define the $\epsilon$-rescaled local dynamics of $f^p$ at $x$
till the time $n$ as the following sequence  of functions $\mathcal{F}^n_x(\epsilon,p)=\left(F_k:B_k\subset E_k\rightarrow E_{k+1}\right)_{k=0,...,[n/p]}$ : for any $0\leq k< [n/p]$, we let  $$\ F_k=\left(\Exp_{f^{(k+1)p}x}(\epsilon.)\right)^{-1}\circ f^p\circ \Exp_{f^{kp}x}(\epsilon .)$$
and $$F_{[n/p]}=\left(\Exp_{f^{n}x}(\epsilon.)\right)^{-1}\circ f^{n-[n/p]p}\circ \Exp_{f^{[n/p]p}x}(\epsilon .).$$
We also let $F^l=F_{l-1}\circ...\circ F_0$ for $0<l\leq [n/p]+1$ and $F^0=\Id$ so  that we have $F^{[n/p]+1}=\left(\Exp_{f^{n}x}(\epsilon.)\right)^{-1}\circ f^n\circ \Exp_{x}(\epsilon .)$. When $p$ divides $n$ the image $\Exp_x\left(B\left(\mathcal{F}^n_x(\epsilon,p),[n/p]\right)\right)$ by $\Exp_x(\epsilon.)$ of the unit $[n/p]$-dynamical ball  for $\mathcal{F}^n_x(\epsilon,p)$ is exactly  the $[n/p]$-dynamical ball $B_{f^p}(x,[n/p],\epsilon)$ at $x$ for $f^p$. In the following the point $x$ will always be assumed to be a periodic point with period $n$, so that we have  $(E_{[n/p]+1}, \| \cdot \|_{[n/p]+1})=(E_0, \| \cdot \|_0)$.

\subsection{Covering dynamical balls with hexagons}

We first generalize the notion of hyperbolic hexagon in the context of a smooth surface dynamical system as follows.
A subset of $M$ is called a local $n$-hyperbolic hexagon (resp. a local generalized $n$-hyperbolic hexagon) for $f$ at $x\in \Per_n$  when it is for some $k$  the image by $f^{-k}\circ \Exp_{f^kx}$ of a  $n$-hyperbolic hexagon  (resp. generalized $n$-hyperbolic hexagon) of  the sequence $\mathcal{F}^n_{f^kx}(\epsilon,p)$.  Obviously   Lemma \ref{ff} again holds true with this generalized definition : there is at most one $n$-periodic point  in a generalized local $n$-hyperbolic hexagon for large enough $n$.

 In the lemma below we consider a Riemannian ($C^r$ smooth) surface $(M,\| \cdot \|)$. We denote again by $\| \cdot \|$ the bundle norm  induced by the Riemannian structure of $M$ on  bundle maps $F:TM\rightarrow TM$ over $f:M\rightarrow M$  and we let $m$ be its conorm, i.e. for any $x\in M$  we let 
 $$\| F(x,.)\|=\sup_{v\in T_xM\setminus \{0\}}\frac{\|F(x,v)\|_{f(x)}}{\|v\|_x}$$ 
 and $$m(F(x,.))=\inf_{v\in T_xM\setminus \{0\}}\frac{\|F(x,v)\|_{f(x)}}{\|v\|_x}.$$ 
We fix once and for all a finite $C^r$ atlas $\mathcal{A}$ of $M$. Then the $C^r$ norm $\|f\|_r$ of any smooth map $f:M\rightarrow M$ with respect to $\mathcal{A}$ (the choice of $\mathcal{A}$ will be implicit later on) is the maximum of the $C^r$ norms $\| \cdot \|_r$ (as defined in Subsection \ref{yy})  of $f$ over the collection of relevant charts.

\begin{lemma}\label{main}
Let $f:M\rightarrow M$ be a $\mathcal{C}^r$  surface diffeomorphism with $r>1$ and let $\delta>0$.

 Then for any integer  $s>r$  and for any positive integer $p$ with $p\delta>>s$,   there exists $\epsilon>0$  depending only on  $\|f^p\|_r$,  such that for all positive integers $n$ and for all $x\in \Per_n^\delta$, there exists a finite collection $\mathcal{J}_n(x)$ of subsets of $M$  with the following properties:

\begin{enumerate}[(i)]
\item any $J\in \mathcal{J}_n(x)$ is  a local generalized $n$-hyperbolic hexagon,\\
\item $\Per_n^\delta\cap B(x,n,\epsilon)\subset \bigcup_{J\in \mathcal{J}_n(x)} J$, \\
\item $\log\sharp \mathcal{J}_n(x) \leq [n/p]\frac{2}{r}\lambda^+_n(x,f^p)+ 2[n/p]\left(\frac{1}{r-1}+\frac{1}{s}\right)\lambda_n(x,f^p)+[n/p]A+B,$
\end{enumerate}
where $A$ is a constant depending only on $r,s$, whereas  $B$ is a constant depending only on $f$, $\delta$ and $p$ (not on $x$ and $n$). Also we have used the notations $$\lambda_n(x,f^p):=\frac{1}{[n/p]}\sum_{j=0}^{[n/p]-1}\log\left(\frac{\|T_{f^{pj}x}f^p\|}{m(T_{f^{pj}x}f^p)}\right) \text{ and }$$
$$\lambda^+_n(x,f^p):=\frac{1}{[n/p]}\sum_{j=0}^{[n/p]-1}\log^+\|T_{f^{pj}x}f^p\|.$$
\end{lemma}

The upper bound in (iii) looks independent from the parameter $\delta$ but this is not the case as $p$ is chosen large compared to $s/\delta$. In fact by the already mentioned works of Kaloshin one can not replace the subset $\Per^\delta
$ of periodic points  by the set $\Per$ of all periodic points in the above statement. \\

We prove now Theorem \ref{key} assuming Lemma \ref{main}. Let $\mu$ be a $f$-invariant probability measure. We denote by $\chi(\mu,f):=\int \lim_n\frac{1}{n}\log \|T_xf^n\|d\mu(x)$ the largest (maybe nonpositive) Lyapunov exponent of $\mu$. Fix $\delta>0$ and $s\in \mathbb{N}$ with $s\gg r$. We let  $p$ be a positive integer such that $$p\delta\gg s,$$ $$\frac{A}{p}\ll 1,$$ 
$$\frac{1}{p}\int \log^+\|T_xf^p\|d\mu(x)\simeq \chi^+(\mu,f)\text{ and }$$ $$\frac{1}{p}\int \log\left(\frac{\|T_xf^p\|}{m(T_xf^p)}\right)d\mu(x)\simeq \chi(\mu,f)+\chi(\mu,f^{-1}).$$

As there is at most one periodic point in $\Per_n^\delta$ in a local generalized $n$-hyperbolic hexagon, we get  for any periodic measure  $\nu_n$ supported on the orbit of a periodic point $x\in \Per_n^\delta$ with minimal period $n$ and for $k$ with $\epsilon_k\leq \epsilon$ : $$\mathfrak{P}_k(\nu_n)\leq \frac{1}{n}\int \log \sharp\mathcal{J}_n(x)d\nu_n(x).$$  Then for  $n$ large enough compared to $B$ and for $s$ large enough compared to $r$ we have
\begin{eqnarray*}
 \frac{1}{n} \int \log\sharp \mathcal{J}_n(x)d\nu_n(x) & \lesssim & \int \left(\frac{2}{pr}\lambda^+_n(x,f^p)+ \frac{2}{p(r-1)}\lambda_n(x,f^p) \right)d\nu_n(x),\\
 &\lesssim & \frac{2}{pr}\int \log^+\|T_xf^p\|d\nu_n(x)+ \frac{2}{p(r-1)}\int \log\left(\frac{\|T_xf^p\|}{m(T_xf^p)}\right)d\nu_n(x).
 \end{eqnarray*}
When $\nu_n$ is going to $\mu$ then by continuity of the integrand this last term is going to
\begin{eqnarray*}
\frac{2}{pr}\int \log^+\|T_xf^p\|d\mu(x)+  \frac{2}{p(r-1)}\int \log\left(\frac{\|T_xf^p\|}{m(T_xf^p)}\right)d\mu(x)\simeq  \ \ \ \ \ \ \ \ \ \ \ \ \ \ \ \ \ \ \ \  \ \ \ \ \ \   \ \ \ \  \ \ \ \ \ \ \ \\
\ \ \ \ \ \ \ \ \ \ \ \ \ \ \ \ \ \ \ \ \ \ \ \ \ \ \ \ \ \ \ \  \ \     \frac{2}{r}\chi^+(\mu,f)+ \frac{2}{(r-1)}\left(\chi(\mu,f)+\chi(\mu,f^{-1})\right),
\end{eqnarray*}

and thus we conclude that
$$g_{\Per^\delta}^*(\mu)=\lim_k\limsup_{\nu_n\rightarrow \mu} \mathfrak{P}_k(\nu_n) \leq  \frac{2(3r-1)}{r(r-1)}\max\left(\chi^+(\mu,f),\chi^+(\mu,f^{-1})\right).$$
This completes the proof of Theorem \ref{key}.

\subsection{ Reparametrization lemma for the tangent map}

Let $\mathcal{F}=(F_k)_{0\leq k \leq  m}$ be a sequence of $C^r$ maps with $r>1$ from the unit Euclidean ball $B\subset\mathbb{R}^2$ to $\mathbb{R}^2$.
We consider the action of the differential of $\mathcal{F}$ on the (trivial) unit fiber bundle $\mathfrak{E}:=\{(x,u,v), \ x\in B, \ (u,v) \in S_x\times S_x\}$ where $S_x=T_x^1\mathbb{R}^2=\mathbb{T}$ is the unit circle. For any $0\leq k \leq m$ and for any $u\in T_x^1M$ we let $u_k=\frac{T_xF^k(u)}{\|T_xF^k(u)\|}\in T^1M$. For two  sequences of  integers $\mathcal{A}=(a_0,...,a_k,...a_{m})\in \mathbb{Z}^{m+1}$ and $\mathcal{B}=(b_0,...,b_k,...b_{m})\in \mathbb{Z}^{m+1}$  we define
 the associated dynamical balls \begin{eqnarray*}
 B_{T\mathcal{F}}(\mathcal{A},\mathcal{B}):=&\bigcap_{0\leq k\leq m} \Big \{ (x,u,v) &\in B\times \mathbb{T} \times \mathbb{T}\ | \   F^kx\in B(0,1), \\
  & & \|T_{F^kx}F_{k}(u_k)\| \in [e^{a_k},e^{a_k+1}], \\  & &    \|T_{F^kx}F_{k}(v_k)\| \in [e^{-b_k-1},e^{-b_k}]\Big\},
 \end{eqnarray*}

\begin{center}and \end{center}
 \begin{eqnarray*}
 B^+_{T\mathcal{F}}(\mathcal{A},\mathcal{B}):=& \bigcap_{0\leq k\leq m} \Big \{ (x,u,v) & \in B\times \mathbb{T} \times \mathbb{T}\ | \  F^kx\in B(0,2), \\
  & &  \|T_{F^kx}F_{k}(u_k)\| \in [e^{a_k-1},e^{a_k+2}],  \\
 & & \|T_{F^kx}F_{k}(v_k)\| \in [e^{-b_k-2},e^{-b_k+1}] \Big \}. \\
 \end{eqnarray*}
By applying Yomdin's reparametrization method to the tangent map we obtain the following lemma (we postpone its proof in the next section).

\begin{lemma}\label{tec}Let  $\mathcal{F}$ be as above. Then  for all  integers $s>r$, for all $\alpha>0$  and for all sequences $\mathcal{A}$, $\mathcal{B}$ of $m+1$ integers there is a family $\mathcal{I}_m=\mathcal{I}_m(\mathcal{A},\mathcal{B})$   of Nash maps  $\psi=(\psi_1,\psi_2,\psi_3):]0,1[^2\times ]0,1[\times ]0,1[\rightarrow B\times \mathbb{T}\times \mathbb{T}$ continuously extendable on $[0,1]^2\times [0,1]\times [0,1]$ such that:

\begin{enumerate}[(i)]
\item Any $\psi$ is a diffeomorphism onto its image and    $\psi(x,v,w)=(\psi_1(x),\psi_2(x,v),\psi_3(x,w))$ for all $(x,v,w)$ (i.e. $\psi_1$; resp. $\psi_2$; resp. $\psi_3$  depends only on $x$; resp. $x,v$; resp. $x,w$). 
\item
$$B_{T\mathcal{F}}(\mathcal{A},\mathcal{B})\subset \bigcup_{\psi\in \mathcal{I}_m}\psi([0,1]^2\times [0,1]\times [0,1])\subset B^+_{T\mathcal{F}}(\mathcal{A},\mathcal{B})$$ and for any $\psi\in \mathcal{I}_m$, for any $k=0,...,m+1$,  the sets $$\bigcup_{(x,v)\in [0,1]^2\times [0,1]}T_{\psi_1(x)}F^{k}(\psi_2(x,v)) \text{ and }\bigcup_{(x,w)\in [0,1]^2\times [0,1]}T_{\psi_1(x)}F^{k}(\psi_3(x,w))$$ are contained in cones of $\mathbb{R}^2$ with  aperture less then $\alpha/2$,

\item \begin{eqnarray*}\sharp \mathcal{I}_m&\leq BA^m\prod_{0\leq k\leq m}&\max \left(  1,\|F_k\|_r^{\frac{1}{r}}, \left(\frac{\|F_k\|_{r}}{e^{a_k}}
\right)^{\frac{1}{r-1}}, \left( \frac{\|F_k\|_{r}}{e^{-b_k}} \right)^{\frac{1}{r-1}} \right)^2\\ & &\times   \max \left(  1,\left(\frac{\|F_k\|_{r}}{e^{a_k}}
\right)^{\frac{1}{s}}\right)\times \max \left(  1,\left( \frac{\|F_k\|_{r}}{e^{-b_k}} \right)^{\frac{1}{s}}\right),
\end{eqnarray*}
\item for any $\psi\in \mathcal{I}_m$ the degree of $\psi$ is less than $A^{m+1}$,\\
\end{enumerate}
where $A$ is a constant depending only on $r,s$, whereas  $B$ depends only on $\alpha$ and on the $C^r$ norms $\| \cdot \|_r$ of  $\mathcal{F}$  (not on $m$).\\
\end{lemma}

Assuming the above Reparametrization Lemma we prove now Lemma \ref{main}. We fix once and for all  $\delta>0$ and $s>r$. We also let $p$ and $n\geq p$ be integers as in Lemma \ref{main} and we consider some $x\in \Per_n^\delta$.

  For $y\in \Per_n^\delta$ we let $v_y$ and $w_y$ be respective representatives in $T^1M$  of  the  unstable and stable directions at the hyperbolic periodic point $y$.  \\

\textit{\textsf{1. Reduction to periodic points with uniform angle.}}
For $x\in M$, ${\kappa}\in \mathbb{Z}$  and a cone $C$ in $T_{f^{\kappa}x}M$   we denote by $T\Exp_{f^{\kappa}x}C$ the subset of $(y,v)\in T^1M$ with $y=\Exp_{f^{\kappa}x}y'$ and $v=T_{y'}\Exp_{f^{\kappa}x}(v')$ for some $y'\in T_{f^{\kappa}x}M$ and $v'\in C$.

\begin{Claim}
It is enough in item (ii) of Lemma \ref{main} to cover by local hyperbolic hexagons  the set $$E:=\left\{z\in \Per_n^\delta\cap B(f^{\kappa}x,n,\epsilon) \ | \ (v_z,w_z)\in (T\Exp_{f^{\kappa}x}C_u(1/2),T\Exp_{f^{\kappa}x}C_s(1/2))\right\}$$   for some $0\leq {\kappa}<n$ and for some bicones $(C_u,C_s)$ of $T_{f^{\kappa}x}M\simeq \mathbb{R}^2$   in $\mathfrak{C}_\alpha$, the number $\alpha$ depending only on $\delta$ and $f$.
\end{Claim}

\label{pag}
 Indeed there is such an $\alpha=\alpha(\delta,f)>0$ that for any $(y,v,w)$  with $y\in B(x,n,\epsilon)$ and  $v,w\in T_y^1M$ satisfying  $\|T_yf^n(v)\|\geq e^{\delta n}, \ \|T_yf^n(w)\| \leq e^{-\delta n}$ the unsigned angle between $T_yf^{\kappa}(v)$ and $T_yf^{\kappa}(w)$ is in $]\alpha, \pi-\alpha[ $ for some $0\leq \kappa< n$ (one proves similarly that the angle between the stable and unstable directions is bounded from below by a constant depending only on $\delta$ and $f$ on a set of $\mu$ -positive measure for any ergodic hyperbolic measure with Lyapunov exponents $\delta$-away from $0$). Also $z=f^\kappa y\in B(f^{\kappa}x,n,\epsilon)$ when $x$ and $y$ are $n$-periodic points and thus the desired cover is obtained by taking the pull back  by $f^\kappa$  of hexagons covering  the above subset $E$ of $\Per_n^\delta\cap B(f^{\kappa}x,n,\epsilon)$. 
 Finally we note  that in the Claim we  may choose  the identification $T_{f^{\kappa}x}M\simeq \mathbb{R}^2$ in such a way $C_u$ is centered at $\{0\}\times \mathbb{R}$.\\

\textit{\textsf{2. Choice of the local dynamics.}} Let $m=[n/p]$. We  consider the sequence  $\mathcal{F}=(F_0,F_1,...,F_m)$
obtained by letting $F_l\in \mathcal{F}_x^n(\epsilon,p)$ for $l\neq 0,m$ and $F_0,F_m$ corresponding to the connected pieces of the orbit of $x$ (with length
less than $2p$) starting and finishing at $f^{\kappa}x$, i.e. if $\kappa\in [Kp,(K+1)p]$ with $K<m$ then we let $$F_m=\left(\Exp_{f^{\kappa}x}(\epsilon.) \right )^{-1}\circ f^{\kappa-Kp}\circ \Exp_{f^{Kp}
x}(\epsilon .) \text{  and }$$
$$F_0=\left(\Exp_{f^{(K+1)p}x}(\epsilon.)\right)^{-1}\circ
f^{(K+1)p-{\kappa}}\circ \Exp_{f^{{\kappa}}x}(\epsilon .)$$  Whenever $K=m$ we replace $(K+1)p$  by $n$ in the definition of $F_0$. For simplicity we may assume $\kappa=0$ and thus $\mathcal{F}=\mathcal{F}_x^n(\epsilon,p)$ in the following without loss of generality.\\

\textit{\textsf{3. Fixing the growth of the derivative.}}
For $y\in \Per_n^\delta\cap B(x,n,\epsilon)$ we denote  by $(y',v'_y,w'_y)$ the image of $(y,v_y,w_y)$ through the local chart at $x$ given by the inverse of the exponential map 	at $x$, that is
$y=\Exp_{x}y'$, $v_y=T_{y'}\Exp_{x}(v'_y)$ and $w_y=T_{y'}\Exp_{x}(w'_y)$ (we recall that  we choose $\epsilon$  less than the radius of injectivity of $M$). Later on we consider the following set $\mathcal E$ of sequences $(\mathcal A,\mathcal B)$:
$$\mathcal{E}:=\{(\mathcal{A},\mathcal{B})\in \mathbb{Z}^{m+1}\times \mathbb{Z}^{m+1}, \ \exists y \in \Per_n^\delta\cap B(x,n,\epsilon) \text{ with } (y',v'_y,w'_y)\in B_{T\mathcal{F}}(\mathcal{A},\mathcal{B})\}.$$
We recall that $H:[1,+\infty[\rightarrow \mathbb{R}$ is defined for all $t\geq 1$ by 
$$H(t)=-\frac{1}{t}\log(\frac{1}{t})-(1-\frac{1}{t})\log (1-\frac{1}{t}).$$
\begin{Claim}There is a constant  $D$ depending only on $f$ and $p$ such that
\begin{eqnarray}\label{comb}
\sharp \mathcal{E}&\leq & 100^{[n/p]}\left(e^{\left(\lambda_n(x,f^p)[n/p]+D\right) H(p\delta/2)}\right)^2.
\end{eqnarray}
\end{Claim}

 Let us prove this claim. We consider  the sequences $(l'_k(y))_{0\leq k\leq m}$ and $(l''_k(y))_{0\leq k\leq m}$ of positive integers  defined for all $ k=0,... ,m$ and for all $y\in \Per_n^\delta\cap B(x,n,\epsilon)$ as follows (recall  $(v_{y'})_k=\frac{T_yF^k(v_{y'})}{\|T_yF^k(v_{y'})\|}\in T^1M$ and similarly $(w_{y'})_k=\frac{T_yF^k(w_{y'})}{\|T_yF^k(w_{y'})\|}\in T^1M$):
$$l'_k(y)=\left[\log\left(\frac{ \|T_{F^k y' }F_k\left((w_{y'})_k\right)\|}{m(T_{F^k y'}F_k)}\right)\right]+1$$
and
$$l''_k(y)=\left[\log\left(\frac{ \|T_{F^ky'}F_k \|}{\|T_{F^ky'}F_k\left((v_{y'})_k\right) \|}\right)\right]+1.$$

Clearly we have
$$\max(l'_k(y),l''_k(y))\leq \left[\log\left(\frac{ \|T_{F^ky'}F_k \|}{m(T_{F^ky'}F_k)}\right)\right]+1.$$
Then we may choose $\epsilon>0$ so small compared to $1/ \|f^p\|_{\min(2,r)}$ that for any $y\in B(x,n,\epsilon)$ we have
$$ \log\left(\frac{ \|T_{F^ky'}F_k \|}{m(T_{F^ky'}F_k)}\right)\leq  \log\left(\frac{ \|T_{0}F_k \|}{m(T_{0}F_k)}\right)+1.$$

Since $x$ is a saddle $n$-periodic point with Lyapunov exponents $\delta$-away from zero  we have
$$\sum_{k=0}^m\log\left(\frac{ \|T_{0}F_k \|}{m(T_{0}F_k)}\right)\geq n\delta.$$
Observe also that $$\left|\sum_{k=0}^m\log\left(\frac{ \|T_{0}F_k \|}{m(T_{0}F_k)}\right)-m\lambda_n(x,f^p)\right|<5\sup_{z\in M, \ 0\leq k\leq 2p}\log\left(\frac{ \|T_{z}f^k \|}{m(T_{z}f^k)}\right).$$
We will use the following combinatorial fact \cite{Burguet}:

\begin{Fact}For any real number $\sigma>0$ the number of sequences of positive integers  $(k_0,...,k_m)$ with $\sum_{i=0}^m k_i \leq (m+1) \sigma$ is less than $e^{(m+1)\sigma H(\sigma)}$.
\end{Fact}
When applying the above Fact to the sequences $(l'_k(y))_{0\leq k\leq m}$ and $(l''_k(y))_{0\leq k\leq m}$ with $$(m+1)\sigma=\sum_{k=0}^m\left(\log\left(\frac{ \|T_{0}F_k \|}{m(T_{0}F_k)}\right)+1\right)$$
we obtain that the number of such sequences of positive integers satisfy  for some constant  $D$ depending only on $f$ and $p$:
\begin{eqnarray*}
\sharp\{(l'_k(y))_{0\leq k\leq m},(l''_k(y))_{0\leq k\leq m}) \}&\leq &e^{2\left(\lambda_n(x,f^p)m+D\right) H(p\delta/2)}.
\end{eqnarray*}\\
Indeed  with  $D=5\sup_{z\in M, \ 0\leq k\leq 2p}\log\left(\frac{ \|T_{z}f^k \|}{m(T_{z}f^k)}\right)$ we have : $$(m+1) \sigma \leq m\lambda_n(x,f^p)+D \text{ and}$$
$$\sigma \geq \frac{n\delta}{m+1}\geq \frac{p\delta}{2}>>1.$$
This implies the inequality (\ref{comb}) of the claim as
 for $\epsilon>0$ small enough we have also $$\left| l'_k(y)-\left[\log T_{F^k y' }F_k\left((w_{y'})_k\right)\right]+\log m(T_{0}F_k) \right|< 3,$$
$$\left|l''_k(y)-\left[\log \frac{1}{\|T_{F^ky'}F_k\left((v_{y'})_k\right) \|}\right]-\log \|T_{0}F_k \| \right|< 3.$$\\

\textit{\textsf{4. The $n$-hyperbolic structure.}}
To obtain the combinatorial Inequality (\ref{comb}) we  just needed that  $x$ belongs to $\Per_n^\delta$ (in the definition of $\mathcal{E}$ we could have considered   general triples $(y',v',w')$ with $y\in B(x,n,\epsilon)$ and $v',w'\in T^1_yM$).   We fix  now sequences $\mathcal{A}=(a_0,...,a_k,...a_{m})$ and $\mathcal{B}=(b_0,...,b_k,...b_{m})$ with $(\mathcal{A},\mathcal{B})\in \mathcal{E}$, in particular we have
  $$\sum_{k}a_k\geq n\delta \text{ and }\sum_kb_k\geq n\delta.$$ Here we have used the fact that the sequences $\mathcal{A}$ and $\mathcal{B}$ are associated to a triple $(y,v_y,w_y)$ with $y\in \Per_n^\delta$ and $v_y,w_y$ being the usual unstable and  stable spaces of the hyperbolic point $y$. The maps $(F_k)_k$ in $\mathcal{F}$ are defined on  the tangent spaces  along the orbit of $x$. By a change of fixed isometric charts, we may assume  $F_k:B\subset \mathbb{R}^2\rightarrow \mathbb{R}^2$ (we did not change the notations for simplicity). Then we  apply Lemma \ref{tec} to $\mathcal{F}$  with $2s$ instead of $s$. Let $\psi=(\psi_1,\psi_2,\psi_3)\in  \mathcal{I}_m$ such that there is $z\in \Per_n^\delta$  with $ (v_z,w_z)\in (T\Exp_{x}C_u(1/2),T\Exp_{x}C_s(1/2))$ and $(z',v'_z,w'_z)$ in the image of $\psi$. Let $(x_z,t_z,s_z)\in [0,1]^2\times [0,1]\times [0,1]$ be such that $(z',v'_z,w'_z)=\psi(x_z,t_z,s_z)$.

  \begin{Claim} The sets $\psi_1(]0,1[^2)$ are  $(m+1)$-hyperbolic sets for $\mathcal{F}$.
  \end{Claim}

   Indeed   the vector fields $(\psi_1(x),\psi_2(x,t_z))_{x\in \psi_1(]0,1[^2)}$ and $(\psi_1(x),\psi_3(x,s_z))_{x\in \psi_1(]0,1[^2)}$ defines an hyperbolic structure on $\psi_1(]0,1[^2)$ by letting $e_s(\psi_1(x))=\psi_2(x,t_z)$ and $e_u(\psi_1(x))=\psi_3(x,s_z)$, because for any $x\in ]0,1[^2$
  we have $\psi(x)\in  B^+_{T\mathcal{F}}(\mathcal{A},\mathcal{B})$ by item $(ii)$ of Lemma \ref{tec} and thus :
\begin{itemize}
\item  $\log \|T_{\psi_1(x)}F^{m+1}\left( \psi_2(x,t_z)\right)\|\geq \sum_{k=0}^{m} (a_k-1)\gtrsim mp\delta,$
\item $ -\log  \|T_{\psi_1(x)}F^{m+1}\left((x,\psi_3(x,s_z)\right)\|\geq  \sum_{k=0}^{m} (b_k-1)\gtrsim mp\delta.$
\end{itemize}

 Moreover since the angle between  $T_{\psi_1(x)}F^{m+1}\left( \psi_2(x,v_z)\right)$ and  $T_{\psi_1(y)}F^{m+1}\left( \psi_2(y,v_z)\right)$, respectively
  between  $T_{\psi_1(x)}F^{m+1}\left( \psi_3(x,w_z)\right)$ and  $T_{\psi_1(y)}F^{m+1}\left( \psi_3(y,w_z\right)$, is less than $\alpha/2$ for any $y\in ]0,1[^2$. But as pointed out in Remark \ref{touto} it is not enough to control the aperture of the two cones given by the images of $e_s$ and $e_u$ under $TF^{m+1}$. We need also to check that these 
  images lie in $C_s$ and $C_u$ respectively. As $z$ is in $\Per_n^\delta$ the vectors  $\psi_2(x_z,t_z)$ and $\psi_3(x_z, s_z)$ are the usual stable and unstable (Oseledets) directions of the saddle hyperbolic periodic point $z$. Therefore they are invariant under $T_{\psi_1(x_z)}F^{m+1}$. Since we have $ (v_z,w_z)\in (T\Exp_{x}C_u(1/2),T\Exp_{x}C_s(1/2))$ we get finally for any $x\in ]0,1[^2$:
  \begin{itemize}
\item $ \psi_2(x,v_z)$ and $T_{\psi_1(x)}F^{m+1}\left( \psi_2(x,v_z)\right)$ lie in $C_s$,
\item $\psi_3(x,v_z)$ and $T_{\psi_1(x)}F^{m+1}\left((x,\psi_3(x,w_z)\right)$ lie in $C_u$. 
\end{itemize}

As already mentioned (see the discussion below Definition \ref{dsd}) we may assume the $n$-expanding field $e_u$ to be constant equal to the oriented center $(0,1)$ of the cone $C_u$. \\

\textit{\textsf{5. Cover by hexagons.}}
For any Nash map  $\psi=(\psi_1,\psi_2,\psi_3)$ as above,
 the set $\psi_1(]0,1[^2)$ may be covered by a collection of generalized $(m+1)$-hyperbolic hexagons for $\mathcal{F}$ with cardinality less than $P\left(\deg(\psi)\right)$ for some universal polynomial $P$ by Lemma \ref{d}. 
According to the  upper bound on the degree of $\psi$ given by item (iv) of Lemma \ref{tec} we have for some universal constant $A$:\\

 \text{\textit{Any  $\psi_1(]0,1[^2)$ may be covered by a collection $\mathcal{C}(\psi)$ of generalized $(m+1)$-hyperbolic   }}
\begin{eqnarray} \label{sem}\textit{hexagons  with }\sharp \mathcal{C}(\psi) \leq A^{m+1}.\end{eqnarray}\\

\textit{\textsf{6. Bounding the cardinality of the cover.}}
 By letting $\epsilon>0$ small enough when defining the local dynamics  $\mathcal{F}_x^n(\epsilon,p)$ we may assume  for all $k$ that:
 $$\|F_k\|_r\simeq \|T_{0}F_k\| \text{ since we have }F_k(0)=0$$
and then $$\max\left(\frac{\|F_k\|_{r}}{e^{a_k}}, \frac{\|F_k\|_{r}}{e^{-b_k}}\right) \lesssim \frac{\|T_{0}F_k\|}{m(T_{0}F_k)},$$
so that  by Lemma \ref{tec} (iii)  we get (the constants $A$ and $B$ may change at each of the following steps):
\begin{eqnarray*}
\sharp \mathcal{I}_m(\mathcal{A},\mathcal{B})&  \leq  BA^m\sup_{(\mathcal{A},\mathcal{B})}\prod_{0\leq k\leq m}&\max \left(  1,\|F_k\|_r^{\frac{1}{r}}, \left(\frac{\|F_k\|_{r}}{e^{a_k}}
\right)^{\frac{1}{r-1}}, \left( \frac{\|F_k\|_{r}}{e^{-b_k}} \right)^{\frac{1}{r-1}} \right)^2\\ 
& &\times   \max \left(  1,\left(\frac{\|F_k\|_{r}}{e^{a_k}}
\right)^{\frac{1}{2s}}\right)\times \max \left(  1,\left( \frac{\|F_k\|_{r}}{e^{-b_k}} \right)^{\frac{1}{2s}}\right),\\
&\leq  BA^m\prod_{0\leq k<m}&\max \left(  1,\|T_0F_k\|^{\frac{1}{r}}, \left(\frac{\|T_0F_k\|}{m(T_{0}F_k)}
\right)^{\frac{1}{r-1}}, \left( \frac{\|T_0F_k\|}{m(T_{0}F_k)} \right)^{\frac{1}{r-1}} \right)^2,\\
& &\times   \max \left(  1,\left(\frac{\|T_0F_k\|}{m(T_{0}F_k)} \right)^{\frac{1}{s}}\right).
\end{eqnarray*}

Lastly observe $[n/p]\lambda^+_n(x,f^p)\simeq \sum_{k}\log^+\|T_0F_k\|$  and similarly
$[n/p]\lambda_n(x,f^p)\simeq  \sum_{k}\log\left(\frac{\|T_0F_k\|}{m(T_{0}F_k)}\right)$. Therefore we obtain

\begin{eqnarray}\label{geom}
\sharp \mathcal{I}_m(\mathcal{A},\mathcal{B}) &\leq  & BA^me^{[n/p]\frac{2}{r}\lambda^+_n(x,f^p)+ 2[n/p]\left(\frac{1}{r-1}+\frac{1}{2s}\right)\lambda_n(x,f^p)}.
\end{eqnarray}\\

\textit{\textsf{7. Conclusion.}}
If $H$ is a generalized $(m+1)$-hyperbolic hexagon for $\mathcal{F}$ then $\exp_x H$ is a local generalized  $n$-hyperbolic hexagon at $x$ for the corresponding $n$-hyperbolic structure. Thus the collection $\mathcal{J}_n(x)$ of local generalized $n$-hyperbolic hexagons given by  $$\mathcal{J}_n(x):=\bigcup_{(\mathcal{A},\mathcal{B})\in \mathcal{E}}\left(\bigcup_{\psi \in \mathcal{I}_m(\mathcal{A},\mathcal{B})}\exp_x \mathcal{C}(\psi)\right)$$ satisfies the conclusion Lemma \ref{main}. Indeed one checks easily that $$\Per_n^\delta\cap B(x,n,\epsilon)\subset \bigcup_{J\in \mathcal{J}_n(x)}  J.$$ Moreover
  by combining Inequalities (\ref{comb}), (\ref{sem}) and (\ref{geom}) and by choosing $p$ so large that $H(p\delta)<1/s$ we get with $A$  depending only on $r,s$ and  $B$  on $f$, $\delta$ and $p$: $$\log\sharp \mathcal{J}_n(x) \leq [n/p]\frac{2}{r}\lambda^+_n(x,f^p)+ 2[n/p]\left(\frac{1}{r-1}+\frac{1}{s}\right)\lambda_n(x,f^p)+[n/p]A+B.$$
This completes the proof of Lemma \ref{main}.

\begin{rem}
The cone $C_u$ being fixed centered at the $y$-axis it would have been equivalent to
reparametrize the dynamical ball \begin{eqnarray*}
\bigcap_{0\leq k\leq m} \Big \{ (x,u) &\in B\times \mathbb{T} \ | &   F^kx\in B(0,1), \|T_{F^kx}F_{k}(u_k)\| \in [e^{a_k},e^{a_k+1}], \\  & &    \|T_{F^kx}F_{k}(v_k)\| \in [e^{-b_k-1},e^{-b_k}\|\Big\},
 \end{eqnarray*}
with $v=(0,1)$ in Lemma \ref{tec}. Even if it involves four dimensional maps we prefer the  general version  given in Lemma \ref{tec}.
\end{rem}

\subsection{Proof of the reparametrization Lemma}
We prove now the reparametrization lemma (Lemma \ref{tec}) for the tangent map stated in the previous subsection. For simplicity we assume $r\in \mathbb{N}\setminus \{0,1\}$. The general case follows the same lines. 

\subsubsection{An adapted space of smooth functions}

 For any integer $s>r$ we let $C^{r,s}_\mathfrak{E}(]0,1[^2\times ]0,1[\times ]0,1[)$ (resp. $D^{r,s}_\mathfrak{E}(]0,1[^2\times ]0,1[\times ]0,1[)$)  be the set of functions
   from $]0,1[^2\times ]0,1[\times ]0,1[$ to $\mathbb{R}^2\times\mathbb{R}^2\times \mathbb{R}^2$ (resp. $\mathbb{R}^2\times \mathbb{R}\times \mathbb{R}$)    of the form $\psi(x,v,w)=\left(\psi_1(x),\psi_2(x,v),\psi_3(x,w)\right)$ for $(x,v,w)\in]0,1[^2\times ]0,1[\times ]0,1[$ where:
\begin{itemize}
\item   the map   $\psi_1$ admits bounded continuous partial derivatives $\partial^{\gamma}\psi_i$ with $\gamma\in \mathbb{N}^2$ satisfying $|\gamma|\leq r$,
\item  for $i\in \{2,3\}$ the maps $\psi_i$ admit bounded continuous partial derivatives $\partial^{\gamma}\psi_i$ with $\gamma=(\gamma_1,\gamma_2)\in \mathbb{N}^2\times \mathbb{N}$ satisfying $|\gamma_1|\leq r-1$ and $|\gamma_1|+\gamma_2\leq s$.
   \end{itemize}
In particular the map $\psi_1$ (resp. $\psi_2$, $\psi_3$) is $C^r$ (resp. $C^{r-1}$). Moreover the maps $\psi_2$, $\psi_3$ and their derivatives w.r.t. $x$ of order less than $r$ are $C^{s-r+1}$ w.r.t. $v$ and $w$. In fact we will consider in the following only such maps $\psi$ where $\psi_2$ and $\psi_3$ and their derivative w.r.t. $x$ are real analytic w.r.t. $v$ and $w$.  In other terms we can take $s$ as large as we want.

    We endow these vector spaces  with the following norm $\|\cdot\|_{r,s}$. For $1\leq k\leq r$ and $0\leq l\leq s $ we first let $p_{k,l}(\psi)$  be the maximum of the supremum norms of  $\partial_x^{\nu} \psi_1$,  $\partial_x^{\alpha_1}\partial_v^{\alpha_2}\psi_2$ and $\partial_x^{\beta_1}\partial_w^{\beta_2}\psi_3$ over all  $\nu,\alpha_1,\beta_1\in \mathbb{N}^2$, $ \alpha_2,\beta_2\in \mathbb{N}$ with $|\nu |= k$,  $|\alpha_1|\leq k-1$, $|\alpha_1|+\alpha_2=l$, $|\beta_1|\leq k-1$ and  $|\beta_1|+\beta_2=  l$. When $k=0$ we let  $p_{0,l}(\psi)$ be the supremum norm of $\psi_1$ (for any $l$).  Then we let $\|\psi\|_{r,s}$ be the norm defined as:
      $$\|\psi\|_{r,s}:=\max_{0\leq k\leq r,0\leq l\leq s} p_{k,l}(\psi).$$
      
 For a $C^s$ map  $\psi$ we recall the $C^s$ norm $\|\psi\|_s$ is given  by the maximum over  $\nu\in \mathbb{N}^4$  with $|\nu|\leq s$ of  the supremum norms of $\partial^\nu\psi$. In particular when a map $\psi$ in $C^{r,s}_\mathfrak{E}(]0,1[^2\times ]0,1[\times ]0,1[)$ is $C^s$ we have $\|\psi\|_{r,s}\leq \|\psi\|_s$.


  \subsubsection{Proof of Lemma \ref{tec} by induction on the length of $\mathcal{F}$}

We fix $\mathcal{F}$, $\mathcal{A}$, $\mathcal{B}$, $\alpha>0$ and $s>r$ as in the statement of the lemma. We consider the
(nonautonomous) cocycle $\mathcal{G}$ over $\mathcal{F}$ given by the  normalized  action of $\mathcal{F}$ on the unit
tangent bundle. More precisely we let  $\mathcal{G}=(G_k)_{0\leq k\leq m}$ be the following  sequence of functions of $\mathfrak{E}$ : 
 for all  $(x,v,w)\in \mathfrak{E}$,  
$$G_k(x,v,w)=(F_kx,v_k,w_k)=\left( F_k(x), \frac{T_xF_k(v)}{\|T_xF_k(v)\|},
\frac{T_xF_k(w)}{\|T_xF_k(w)\|}\right).$$ 
Also we define  the iterated sequence $(G^k)_{0\leq k\leq m+1}$ 
inductively by $$G^0=\Id$$ and for $0\leq k\leq m $ and for all $(x,v,w)\in \mathfrak{E}$ with $x\in \bigcap_{l=0}^kF^{-l}B$, 
$$G^{k+1}(x,v,w)=G_{k}\circ G^k(x,v,w).$$

By an induction on the length  $m$ of the sequence $\mathcal{F}$ we  will prove the following claim.
\begin{Claim}
For  any $m\in \mathbb{N}\cup\{-1\}$ and for any sequences $\mathcal{A}$ and $\mathcal{B}$ of $m+1$   integers,  there is   a family of   Nash maps $\Psi_m=\{\psi^m\}$ with $\psi^m:]0,1[^2\times ]0,1[\times ]0,1[\rightarrow B\times\mathbb{T}\times \mathbb{T}$ 
and a universal constant $A'$ such that:
\begin{enumerate}
\item $\psi^m(x,v,w)=\left(\psi^m_1(x),\psi^m_2(x,v),\psi^m_3(x,w)\right)$, 
\item   for any $0\leq k\leq m+1$, the map $G^k\circ \psi^m$ belongs to $C^{r,s}_\mathfrak{E}(]0,1[^2\times ]0,1[\times ]0,1[)$ and
 $$\|G^k\circ \psi^m\|_{r,s}\leq 1,$$
\item $$B_{T\mathcal{F}}(\mathcal{A},\mathcal{B}) \subset \bigcup_{\psi^m\in \Psi_m}\psi^m([0,1]^2\times [0,1]\times [0,1])\subset  B^+_{T\mathcal{F}}(\mathcal{A},\mathcal{B}),$$
\item \begin{eqnarray*}
\sharp \Psi_m & \leq
A'&\max \left(  1,\|F_{m}\|_r^{\frac{1}{r}}, \left(\frac{\|F_{m}\|_{r}}{e^{a_{m+1}}}
\right)^{\frac{1}{r-1}}, \left( \frac{\|F_{m}\|_{r}}{e^{-b_{m+1}}} \right)^{\frac{1}{r-1}} \right)^2\\ 
&  &\times \max\left(1,\left(\frac{\|F_{m}\|_{r}}{e^{a_{m+1}}}
\right)^{\frac{1}{s}}\right)\times\max\left(1,\left( \frac{\|F_{m}\|_{r}}{e^{-b_{m+1}}} \right)^{\frac{1}{s}} \right)\\
& &\times\sharp \Psi_{m-1}.
\end{eqnarray*} 
\item for any map $\psi^m\in \Psi_m$ there is $\psi^{m-1}\in \Psi_{m-1}$ and a Nash map $\phi_m:]0,1[^2\times ]0,1[\times ]0,1[\rightarrow ]0,1[^2\times ]0,1[\times ]0,1[$ with $\deg(\phi_m)\leq A' $ such that
$\psi^m=\psi^{m-1}\circ \phi_m$. Moreover for any $\psi^{-1}\in \Psi_{-1}$ we have $\deg(\psi^{-1})\leq A'$. 
\end{enumerate}
\end{Claim}

The above claim implies easily  Lemma \ref{tec}  as follows. We consider   a fixed cover $\mathcal{U}_\alpha$ of $]0,1[^4= ]0,1[^2\times ]0,1[\times ]0,1[$ by subcubes  with diameter less than $\alpha/\pi$. For any such subcube $S\in \mathcal{U}_\alpha$ we let  $\eta_{S}:]0,1[^4\rightarrow S$ be the affine (homothetic) reparametrization of $S$. Then
the map  $G^k\circ \psi^m$ being $1$-Lipschitz for any $0\leq k\leq m+1$ the image of $G^k\circ \psi^m\circ \eta_S$ is contained in a subcube with diameter less than $\alpha/\pi$. But for any $u,v\in \mathbb{T}$ we have
\begin{eqnarray*}
\|u-v\|&=&2 \left|\sin \left(\frac{\angle (u,v)}{2}\right)\right|,\\
&\geq & 2\frac{|\angle (u,v)|}{\pi}.
\end{eqnarray*}
so that with  $\psi=(\psi_1,\psi_2,\psi_3)=\psi^m\circ \eta_S$  the sets $$\bigcup_{(x,v)\in [0,1]^2\times [0,1]}T_{\psi_1(x)}F^{k}(\psi_2(x,v)) \text{ and }\bigcup_{(x,w)\in [0,1]^2\times [0,1]}T_{\psi_1(x)}F^{k}(\psi_3(x,w))$$ are contained in cones of $\mathbb{R}^2$ with  aperture less then $\alpha/2$.
Together with item (3) of the Claim we get item (ii) of Lemma \ref{tec} by letting $\mathcal{I}_m:=\{\psi^m\circ \eta_S,\ \psi^m\in \Psi_m\text{ and }S\in \mathcal{U}_\alpha\}$. Finally  :
\begin{itemize}
\item for all nonnegative integers $m$ we have \begin{eqnarray*}
\sharp \mathcal{I}_m&\leq & \sharp \mathcal{U}_\alpha \times \sharp \Psi_m,\\
& \leq &  \sharp \mathcal{U}_\alpha \times \sharp \Psi_{m-1} \times 
A' \max \left(  1,\|F_{m}\|_r^{\frac{1}{r}}, \left(\frac{\|F_{m}\|_{r}}{e^{a_{m+1}}}
\right)^{\frac{1}{r-1}}, \left( \frac{\|F_{m}\|_{r}}{e^{-b_{m+1}}} \right)^{\frac{1}{r-1}} \right)^2\\ 
&  &\times \max\left(1,\left(\frac{\|F_{m}\|_{r}}{e^{a_{m+1}}}
\right)^{\frac{1}{s}}\right)\times\max\left(1,\left( \frac{\|F_{m}\|_{r}}{e^{-b_{m+1}}} \right)^{\frac{1}{s}} \right)\\
&\leq &\sharp \mathcal{U}_\alpha \times\sharp \Psi_{-1}\times  A'^{m+1}  \prod_{0\leq k\leq m}\max \left(  1,\|F_k\|_r^{\frac{1}{r}}, \left(\frac{\|F_k\|_{r}}{e^{a_k}}
\right)^{\frac{1}{r-1}}, \left( \frac{\|F_k\|_{r}}{e^{-b_k}} \right)^{\frac{1}{r-1}} \right)^2,\\
& & \ \ \ \ \ \  \ \ \ \ \ \ \ \ \  \ \ \ \ \ \ \ \ \ \ \ \ \ \ \ \ \times \max\left(1,\left(\frac{\|F_k\|_{r}}{e^{a_k}}
\right)^{\frac{1}{s}}\right) \times \max\left(1, \left( \frac{\|F_k\|_{r}}{e^{-b_k}} \right)^{\frac{1}{s}}\right).
\end{eqnarray*}
so that we obtain the estimate (iii) on the number of  reparametrization maps in Lemma \ref{tec}  by letting $B=\sharp \mathcal{U}_\alpha \times\sharp \Psi_{-1}$ and $A\geq A'$,
\item  for all $\psi=\psi^m\circ \eta_S\in \mathcal{I}_m$ there are $\psi_{m-1}\in \Psi_{m-1}$ and a Nash map $\phi_m$ with $\deg(\phi_m)\leq A'$ (where $A'$ is a universal constant) and  $\psi^m=\psi^{m-1}\circ \phi_m$ so that:
\begin{eqnarray*}
\deg(\psi)&=&\deg(\psi^m\circ \eta_S),\\
&=& \deg(\psi^m),\\
&=&\deg(\psi^{m-1}\circ \phi_m),
\end{eqnarray*}
then by Corollary \ref{compo}:
\begin{eqnarray*}
\deg(\psi)&\leq & \deg(\phi_m)^4 \times \deg(\psi^{m-1}),\\
&\leq & A'^4  \deg(\psi^{m-1}),\\
&\leq & A'^{4(m+1)}\times \max_{\psi^0\in\Psi_{-1}}\deg(\psi^{-1}),\\
&\leq & A'^{4m+5}.
\end{eqnarray*}
so that we get the upper bound $\deg(\psi)\leq A^{m}$ in Lemma \ref{tec} (iv) on the degree of the reparametrization maps by  letting $A\geq A'^5$. 
\end{itemize}
This concludes the proof of Lemma \ref{tec}.

\subsubsection{Proof of the induction}

Now we go back to the proof of the claim which follows from Yomdin's approach : we first renormalized the reparametrization maps to kill the derivatives of order $r$ and then we apply Yomdin-Gromov Lemma to the Taylor-Lagrange interpolating polynomials of degree $r-1$ to conclude. \\

Let  $\mathcal{A}=(a_k)_{k\in \mathbb{N}}$ and $\mathcal{B}=(b_k)_{k\in \mathbb{N}}$ be two sequences of integers. By arguing inductively  on $m$ we build the family of Nash maps $\Psi_m$ for the sequences   $\mathcal{A}^m$ and $\mathcal{B}^m$ given respectively by the $m+1$-first terms of  $\mathcal{A}$ and $\mathcal{B}$.\\

\textit{\textsf{1. The base case of the induction.}} We let $\mathfrak{A}=(\zeta_1,..., \zeta_p)$ be a Nash
atlas of $\mathbb{T}$, i.e. the maps $\zeta_l$, $1\leq l\leq p$, are Nash diffeomorphisms from $]-1,1[$ to $\mathbb{T}$ with $\bigcup_{l=1,...,p}\zeta_l(]-1,1[)=\mathbb{T}$.

For the family $\Psi_{-1}=\{\psi^{-1}\}$ it is enough to let $\psi^{-1}=(\psi_1^{-1},\psi_2^{-1},\psi_3^{-1})$ be such that  $\psi_1^{-1}$ are homotheties 
 covering the unit ball $B$ whereas $\psi_2^{-1}(x,v)$ and $\psi_3^{-1}(x,w)$ do not depend on $x$ and are equal to  $\zeta_{i}(v)$ and $\zeta_{j}(w)$  for $1 \leq i,j\leq p$. 
  
    We assume now the collection  $\Psi_m$ already built for some integer $m\in \mathbb{N}\cup\{-1\}$. \\

\textit{\textsf{2. Killing the highest derivatives by homothetic renormalizations.}} The tangent map $TF_{m+1}$ acts naturally on $\mathfrak{E}$ by letting for any $(x,v,w)\in \mathfrak{E}$:  $$TF_{m+1}(x,v,w)=(F_{m+1}x, T_xF_{m+1}v, T_xF_{m+1}w).$$  

As the map $\rho:=G^{m+1}\circ \psi^m$ belongs to $C^{r,s}_\mathfrak{E}(]0,1[^2\times ]0,1[\times ]0,1[)$ it has the form 
$\rho(x,v,w)=\left(\rho_1(x),\rho_2(x,v),\rho_3(x,w)\right)$ for $(x,v,w)\in]0,1[^2\times ]0,1[\times ]0,1[$.
In this case the Faa di Bruno  formula  the derivatives of the composition   $TF_{m+1} \circ G^{m+1}\circ \psi^m$ gives for some universal polynomial $R$ depending only on $r$ and $s$:
$$\|TF_{m+1}\circ G^{m+1}\circ \psi^m\|_{r,s}\leq \|F_{m+1}\|_{r} R\left(\|G^{m+1}\circ \psi^m\|_{r,s}\right).$$

By induction hypothesis we have $\|G^{m+1}\circ \psi^m\|_{r,s}\leq 1$. Thus for some constant $C=C(r,s)$ we get
\begin{eqnarray}\label{zzz}\|TF_{m+1}\circ G^{m+1}\circ \psi^m\|_{r,s}&\leq & C \|F_{m+1}\|_{r}.
\end{eqnarray}

We cover the unit square $]0,1[^2$ into isometric open subsquares
$\{\mathcal S\}$ with diameter less than $$\min \left(
 \left(\frac{1}{C\|F_{m+1}\|_{r}}\right)^{\frac{1}{r}} ,
 \left(\frac{ e^{a_{m+1}}}{C\|F_{m+1}\|_{r}} \right)^{\frac{1}{r-1}},
 \left( \frac{ e^{-b_{m+1} } }{C\|F_{m+1}\|_{r}}
 \right)^{\frac{1}{r-1}} \right). $$  Similarly we cover  the second and third factor $]0,1[$   by isometric open  subintervals $\{\mathcal R\}$ and $\{\mathcal T\}$ with diameter  respectively less than
 $$  \left(\frac{ e^{a_{m+1}}}{C\|F_{m+1}\|_{r}}
\right)^{\frac{1}{s}} \text{ and }\left( \frac{ e^{-b_{m+1}}}{C\|F_{m+1}\|_{r}} \right)^{\frac{1}{s}}.$$ 
 
 For any $\psi=(\psi_1,\psi_2,\psi_3)\in  C^{r,s}_\mathfrak{E}(]0,1[^2\times ]0,1[\times ]0,1[)$ and $(u,v,w)\in (\mathbb{R}^+)^3$ we denote  by  $(u,v,w)\cdot \psi$ the function $(u\psi_1,v\psi_2,w\psi_3)$. Then for any $\mathcal{S},\mathcal{R},\mathcal T$ we let $\phi^{m+1}=(\phi_1^{m+1},\phi_2^{m+1},\phi_3^{m+1})$ be the coordinatewise affine (homothetic) reparametrization $\phi^{m+1}:]0,1[^2\times]0,1[\times ]0,1[\rightarrow \mathcal{S}\times \mathcal{R}\times \mathcal{T} $ so that
\begin{eqnarray*}
I&:= & p_{r,s}\left((1,e^{-a_{m+1}},e^{b_{m+1}})\cdot{}(TF_{m+1}\circ G^{m+1}\circ \psi^m\circ \phi^{m+1})\right),\\
 & \leq & p_{r,s}\left(\left(\mathfrak{a}_m,\mathfrak{b}_m,\mathfrak{c}_m
\right)\cdot  ( TF_{m+1}\circ G^{m+1}\circ \psi^m\right),
\end{eqnarray*}
with $\mathfrak{a}_m,\mathfrak{b}_m,\mathfrak{c}_m$ defined as follows
\begin{itemize}
\item $\mathfrak{a}_m=\frac{1}{C\|F_{m+1}\|_{r}}$,
\item $\mathfrak{b}_m=\max_{\alpha,\beta}\frac{\left(\min(1, \mathfrak{a}_m e^{a_{m+1}})\right)^{\beta/s+|\alpha|/(r-1)}}{e^{a_{m+1}}}$,\\
\item $\mathfrak{c}_m=\max_{\alpha,\beta}\frac{\left(\min(1, \mathfrak{a}_m e^{-b_{m+1})}\right)^{\beta/s+|\alpha|/(r-1)}}{e^{-b_{m+1}}}$,
\end{itemize}
where the maxima holds over $(\alpha, \beta)\in \mathbb{N}^2\times \mathbb{N}$ with $|\alpha|\leq r-1$ and $|\alpha|+\beta=s$.
For such pairs $(\alpha,\beta)$ we have $\frac{\beta}{s}+\frac{|\alpha|}{r-1}\geq \frac{\beta+|\alpha|}{s}\geq 1$. Thus we get
$$\mathfrak{b}_m\leq \mathfrak{a}_m \text{ and } \mathfrak{c}_m\leq \mathfrak{a}_m,$$
and therefore by Inequality (\ref{zzz}):
$$I\leq p_{r,s}\left(\frac{ TF_{m+1}\circ G^{m+1}\circ \psi^m}{C\|F_{m+1}\|_{r}} \right)\leq 1.$$\\

\textit{\textsf{3. Taylor-Lagrange approximation.}}
In our setting the multivariate Taylor-Lagrange Inequality may be written as follows for $\psi=(\psi_1,\psi_2,\psi_3)\in C^{r,s}_\mathfrak{E}(]0,1[^2\times ]0,1[\times ]0,1[)$ at $(x_0,v_0,w_0)\in ]0,1[^2\times ]0,1[\times ]0,1[$:
\begin{eqnarray*}
\|\psi(x,v,w)-\sum_{\nu\in \mathbb{N}^2,\ |\nu|=r-1}\frac{(x-x_0)^\nu}{\nu !}(\partial^\nu\psi_1(x_0),0,0)& &\\
 -\sum_{(\alpha,\beta,\gamma)}\frac{(x-x_0)^\alpha(v-v_0)^\beta(w-w_0)^\gamma}{\alpha ! \beta ! \gamma !} \partial_x^\alpha\partial_v^\beta\partial_w^\gamma\psi(x_0,v_0,w_0)\|_{r,s}& \leq C' p_{r,s}(\psi),&
\end{eqnarray*}
where the last sum holds over all $(\alpha,\beta,\gamma)\in \mathbb{N}^2\times\mathbb{N}\times
\mathbb{N}$ with $|\alpha|<r-1$, $|\alpha|+\beta<s$ and $|\alpha|+\gamma<s$ and where $C'$ depends only on $r$ and $s$ (for $r\notin \mathbb{N}$ we should here consider the Taylor-Lagrange polynomial of degree $[r]$). One may change the constant $C$ in the previous paragraph 2., so that we have  $$I=p_{r,s}\left((1,e^{-a_{m+1}},e^{b_{m+1}})\cdot{}(TF_{m+1}\circ G^{m+1}\circ \psi^m\circ \phi^{m+1})\right)\leq 1/C'.$$ Thus there is a polynomial $P=(P_1,P_2,P_3):\mathbb{R}^2\times \mathbb{R} \times \mathbb{R}\rightarrow \mathbb{R}^2\times\mathbb{R}^2\times \mathbb{R}^2$     with $P(x,v,w)=(P_1(x),P_2(x,v),P_3(x,w))$ for all $(x,v,w)\in \mathbb{R}^2\times \mathbb{R}\times \mathbb{R}$ and with $\deg_t(P_1)< r $ and $\deg_t(P_2),\deg_t(P_3)< s$  such that
$$II:=\|P-(1,e^{-a_{m+1}},e^{b_{m+1}})\cdot{}\left(TF_{m+1}\circ G^{m+1}\circ \psi^m\circ \phi^{m+1}\right)\|_{r,s}\leq 1.$$\\

\textit{\textsf{4. Applying Yomdin-Gromov Lemma.}}
 By applying the Yomdin-Gromov algebraic lemma to $P$ for the smoothness parameter $s$ (see Remark \ref{yot} below), we get Nash maps $\theta:]0,1[^2\times ]0,1[
 \times ]0,1[\rightarrow ]0,1[^2 \times ]0,1[\times ]0,1[$ reparametrizing the  set
 \begin{eqnarray*}
  & P^{-1}\left( B(0,2)\times B(0,1+e)\times B(0,2)\right) &  \\
  & \bigcup  & \\
 & \left(TF_{m+1}\circ G^{m+1}\circ \psi^m\circ \phi^{m+1}\right)^{-1}\left( B(0,1)\times  B(0,e^{a_{m+1}+1}) \times B(0,e^{-b_{m+1}}) \right) &  \\
& \bigcup &\\
& \left(\psi^m\circ \phi^{m+1}\right)^{-1}\left( B_{T\mathcal{F}}(\mathcal{A},\mathcal{B})\right). &
 \end{eqnarray*}
  Moreover the maps $\theta$ have the form $\theta
 (t,s,u)=(\theta_1(t),\theta_2(t,s),\theta_3(t,u))$ for all $(t,s,u) $ where
  the degree of $
 \theta$ is bounded by some universal constant $A=A(s)$. Also we have $\theta\in D^{r,s}_\mathfrak{E}(]0,1[^2\times ]0,1[\times ]0,1[)$ and $P\circ \theta\in C^{r,s}_\mathfrak{E}(]0,1[^2\times ]0,1[\times ]0,1[)$ with 
 $$\|\theta\|_{r,s}\leq \|\theta\|_{s}\leq  1 \text{ and }
 \|P\circ \theta\|_{r,s}\leq  \|P\circ \theta\|_{s} \leq 1.$$
 By applying again Faa di Bruno formula we get for some universal polynomial $R$ depending only on $r$ and $s$:
 \begin{eqnarray*}
 III&:=& \|(1,e^{-a_{m+1}},e^{b_{m+1}})\cdot{}\left(TF_{m+1}\circ G^{m+1}\circ \psi^m\circ \phi^{m+1}\circ \theta \right)\|_{r,s},\\
 &\leq & \|\left(P-(1,e^{-a_{m+1}},e^{b_{m+1}})\cdot{}\left( TF_{m+1}\circ G^{m+1}\circ \psi^m\circ \phi^{m+1} \right)\right) \circ \theta \|_{r,s}+\|P\circ \theta\|_{r,s},\\
 &\leq & R\left(II, \|\theta\|_{r,s}\right) +\|P\circ \theta\|_{r,s}.
 \end{eqnarray*}
 Therefore there is a universal constant $C=C(r,s)$ such that
 \begin{eqnarray}\label{n}
 & & III=\|(1,e^{-a_{m+1}},e^{b_{m+1}})\cdot{}\left(TF_{m+1}\circ G^{m+1}\circ \psi^m\circ \phi^{m+1}\circ \theta\right) \|_{r,s}\leq C.
 \end{eqnarray}
By composing with  $D/C$-homotheties with $D= 1-1/e$  (we already used this trick, e.g. to prove Lemma \ref{tec} from the Claim) we may assume that the $i^{th}$ component, for $i= 2$ and $3$,  of   $TF_{m+1}\circ G^{m+1}\circ \psi^m\circ \phi^{m+1}$   are respectively  $De^{a_{m+1}}$- and $De^{-b_{m+1}}$-Lipschitz. In particular as $D=1-1/e<1<e-1$ the image of $\psi^m\circ \phi^{m+1}\circ \theta$ is contained in $B^+_{T\mathcal{F}}( \mathcal{A}^{m+1},\mathcal{B}^{m+1})$  and 
 \begin{eqnarray}\label{derner} \|\left( TF_{m+1}\circ G^{m+1}\circ \psi^m\circ \phi^{m+1}\circ \theta\right)_i\|_0  \leq &  & \\ \nonumber
\ \ \ \ \ \ \ \  \ \ \  \ \   2\inf_{(t,s)}  \|\left( TF_{m+1}\circ G^{m+1}\circ \psi^m\circ \phi^{m+1}\circ \theta\right)_i(t,s)\|.& &
 \end{eqnarray}
 Moreover these components satisfy also for $\alpha\in \mathbb{N}^2$ and $\beta\in \mathbb{N}$ with $|\alpha|\leq r-1$  and $0\neq |\alpha|+\beta\leq s$:
 
 \begin{eqnarray}\label{derniier}\sup_{(t,s)} \|\partial^{(\alpha,\beta)}\left( TF_{m+1}\circ G^{m+1}\circ \psi^m\circ \phi^{m+1}\circ \theta\right)_i(t,s)\|\leq & \\ \nonumber
\ \ \ \ \ \ \ \ \ \ \ \ \ \ \ \ \ \ \ \ \     \inf_{(t,s)}  \|\left( TF_{m+1}\circ G^{m+1}\circ \psi^m\circ \phi^{m+1}\circ \theta\right)_i(t,s)\|.&
 \end{eqnarray}

\textit{\textsf{5. Control of the $\| \cdot \|_{r,s}$-norm of the normalized action.}} We may also control the derivatives of the normalized action according to the following fact:

 \begin{Fact}There is a homogeneous universal polynomial $R\in \mathbb{R}[X_0,...,X_r]$,  such
that for any non vanishing $C^r$  smooth map $u$ we have   for any $\alpha$ with $|\alpha|=r$: 
$$\sup_x\| \partial^\alpha (u/\|u\|)(x)\|\leq \frac{R(\|u\|_0,\|u\|_1,..., \|u\|_r)}{\inf_x \|u(x)\|^{\deg_t(R)}}.$$
 \end{Fact}

 Together with Equation (\ref{derner}) and Equation (\ref{derniier}) we indeed obtain for another constant $C=C(r,s)$:
$$\|G^{m+2}\circ \psi^m\circ \phi^{m+1}\circ \theta \|_{r,s} \leq C.$$
By composing again by  $1/C$-homotheties we may assume $C=1$.
 Then we let $\Psi_{m+1}$ be the collection of maps $\Psi_{m+1}=\{
 \psi^m\circ \phi^{m+1}\circ \theta\}$. One checks easily that this family  satisfies the properties of the Claim. This concludes the proof of the Reparametrization Lemma \ref{tec}.

  \begin{rem}\label{yot}
  In the above proof we apply Yomdin-Gromov Lemma to a $4$-dimensional Nash map $f:]0,1[^2\times ]0,1[\times]0,1[\rightarrow \mathbb{R}^2\times \mathbb{R}^2\times \mathbb{R}^2$ of the form $$f(x,v,w)=(f_1(x),f_2(x,v),f_3(x,w)).$$ This follows from the version stated in  Subsection \ref{yy}. Indeed we may first apply
  this last version to the map $(x,v)\mapsto (f_1(x),f_2(x,v))$. Then if we let $\phi:]0,1[^2\times ]0,1[\rightarrow ]0,1[^2\times ]0,1[$ with $\phi(x,v)=(\phi_1(x),\phi_2(x,v))$ be the triangular  Nash reparametrization maps we apply again Theorem \ref{firs} to  each map $(x,w)\mapsto (f_1\circ \phi_1 (x) ,f_3(\phi_1(x),w))$. Denoting by $\phi'=(\phi'_1,\phi'_2)$ the resulting reparametrization maps  it is then enough to consider the family  $$\{(x,v,w)\mapsto (\phi_1\circ\phi'_1(x), \phi_2(\phi'_1(x),v), \phi'_2( x,w))\}.$$
  \end{rem}

\begin{rem} From the proof we may in fact explicit the dependence in $\delta$ of $\sharp \Per_n^\delta$. Let us be more precise for a $C^\infty$ surface diffeomorphisms $f$. For any $\gamma>0$ there exists an integer $N=N(f,\delta,\gamma)$ such that for any $n$ larger than $N$ we have
$$\sharp Per_n^\delta\leq \alpha(\delta,f)^{-3} e^{\gamma n}$$
where $\alpha(\delta,f)$ is the angle  as in the proof of Lemma \ref{main} assuming Lemma \ref{tec} on page \pageref{pag}.\\
\end{rem}
\vspace{1cm}
\newpage
\begin{center}

\begin{huge}
Appendices \\
\end{huge}
\end{center}
\vspace{1cm}
\appendix

\section{Growth of $(Per_n^\delta)_n$}
\subsection{ The case of $C^r$ surface diffeomorphisms with $r>1$}
Let $f:M\rightarrow M$ be a $C^r$ with $r>1$ surface diffeomorphism.  For any $\chi>0$ Sarig build a finite to one Markov extension $\pi_\chi:(\Sigma(\mathcal{G}_\chi),\sigma)\rightarrow (M,f)$ such that $
\pi_\chi\left(\Sigma(\mathcal{G}_\chi)\right)$ has full measure for any hyperbolic ergodic invariant measure with Lyapunov exponents $\chi$-away from zero. In fact  for any $\alpha<1$ we may choose the parameters in the construction of $\mathcal{G}_\chi$ so that we have:

\begin{lemma}
$$\pi_\chi\left(\Per(\Sigma(\mathcal{G}_\chi),\sigma)\right)\subset \Per_{\alpha \chi}(M,f).$$
\end{lemma}

\begin{proof}
Any $n$-periodic point in Sarig's graph is given by a closed finite (bi)chain of so called overlapping charts
$\Psi:=(\psi^+_1,\psi^-_1)\rightarrow (\psi^{+}_2,\psi^{-}_2)...\rightarrow (\psi^{+}_n,\psi^{-}_n)=(\psi^+_1,\psi^-_1)$.

By Proposition 3.4 in \cite{sar} the gradient of $f_i:=\left(\psi^+_{i+1}\right)^{-1}\circ f \circ \psi^+_{i}$ 
takes the form $$f_i(u,v)=
\left(Au+h_1(u,v), Bv+h_2(u,v)\right)$$
with $A<e^{-\chi}$, $B>e^{\chi}$ and $\|\nabla h_j\|<\epsilon$ for $j=1,2$ where $\epsilon$ may be chosen small compared to $\chi$ and $1-\alpha$.
In particular for $\epsilon$ small enough the differential map $T f_i$  preserves the cone $\{(u,v)\in \mathbb{R}^2, \ |v| \geq |u|\}$ and   $\|Tf_i(u,v)\|'> e^{\alpha \chi}\|v\|'$ for all $(u,v)$ in this cone with the product norm $\|(x,y)\|'=\max(|x|,|y|)$. In particular the largest Lyapunov exponent at $x=\pi_\chi\left(\Psi\right)$ is bigger then $\alpha\chi$.
\end{proof}

It follows from Gurevic's theory (e.g. see the proof of Theorem 1.1 in \cite{sar}) that whenever $f$ has a measure of maximal entropy  then  for some positive integer $p$
  we have  $\liminf_{p\mid n} \sharp\Per_n^\delta e^{ -n\htop(f)}>0$ for any $0<\delta<\htop(f)$. In particular these lower bounds holds for $C^\infty$ surface diffeomorphisms.\\

\vspace{1cm}

\subsection{The case of $C^{r}$ interval maps with $r>1$}
We need here to introduce a slightly modified version of Buzzi-Hofbauer diagram. This diagram is built from the symbolic dynamic associated to a partition into monotone branches. We do not recall the whole construction and refer the reader to the original paper \cite{Buz} for details. Therein  the author used the "natural partition" given by the complementary set of the critical points. However we may also work with other partitions. Let $\htop(f)>\delta>0$ and $\alpha>0$.
For $C^{r}$ interval maps    $ f$ there exists  a countable partition  $P=P_\alpha$ of $[0,1]\setminus \{f'=0\}$ into intervals $I$ so that    \cite{Buzzz}
\begin{eqnarray}\label{ddd}
\forall  x,y\in I, \ & \frac{|f'(x)|}{|f'(y)|}\leq &e^{\alpha}.
\end{eqnarray}

Then we recall that in the  Markov representation associated to this partition we let an arrow from the vertex $A_{-q}...A_0\in P^{q+1}$ to the vertex $B_{-p}...B_0\in P^{p+1}$ whenever  $p\leq q+1$ and $B_{-p}...B_1=A_{-p+1}...A_0$ and $B_0\cap f^{q+1}\left(\bigcap_{k=0}^qf^{k-q}A_k\right)=f^p\left(\bigcap_{k=0}^pf^{k-p}B_k\right)$.

Buzzi proved that there is a Borel semiconjugation from the Markov shift to the interval map which induces  an entropy-preserving bijection between measures of positive entropy (see also \cite{bbu}). In particular if the interval map admits a maximal measure then so does the Markov shift. In this case applying  the already mentioned works of Gurevic  there is a vertex $\gamma$ in the graph so that the number $a_n$ of closed paths of length $n$ at $\alpha=A_{-q}...A_0$  satisfies for some $p$
\begin{eqnarray}\label{der}
\liminf_{p\mid n}a_n e^{-n \htop(f)}&>&0.
\end{eqnarray}
By letting  $I_\gamma:=f^{q}\left(\bigcap_{k=0}^qf^{k-q}A_k\right)$ these paths corresponds to disjoints $n$-monotone branches $I\subset I_\gamma$ with $I_\gamma\subset f^n(I)$. The number $a'_n$ of the intervals $I$ with  length less than $e^{(\alpha-1)n\htop(f)}|I_\gamma|$ also satisfies $\liminf_{p\mid n}a'_n e^{-n \htop(f)}>0$.
For these intervals $I$ we have
$\sup_{x\in I}|(f^n)'(x)|\geq e^{(1-\alpha)n\htop(f)}$ and therefore according to the distortion property (\ref{ddd}) we get
$$\inf_{x\in I}|(f^n)'(x)|\geq e^{\left((1-\alpha)\htop(f)-\alpha\right)n}.$$ Then
any $n$-periodic point in $I$ has a Lyapunov exponent larger $e^{(1-\alpha)\htop(f)-\alpha}$ which is larger than $\delta$ if one takes $\alpha$ small enough. Thus we conclude that
$$\forall n\in \mathbb{N}\setminus  \{0\}, \ \sharp Per_n^\delta\geq a'_n$$ 
and thus
$$\liminf_{p\mid n} \sharp \Per_n^\delta e^{-n \htop(f)}>0.$$

\newpage
\section{Variational principle for the local periodic growth}
We prove here the variational principle (\ref{varia}) stated in Section 2:
\begin{eqnarray}g_\mathcal{P}^*=\sup_\mu g^*_\mathcal{P}(\mu).
\end{eqnarray}

\begin{proof}
Clearly we have for all $k$ and for all ergodic periodic measures $\nu_n$ with minimal period $n$: $$\mathfrak{P}_k(\nu_n)\leq \frac{1}{n}\sup_{x\in X}\log  \sharp \left( \mathcal{P}_n\cap B(x,n,\epsilon_k)\right),$$
and thus $g_\mathcal{P}^*\geq \sup_\mu g^*_\mathcal{P}(\mu).$ Let us show now the other inequality. Fix $\alpha>0$. We will prove that $g_\mathcal{P}^*\leq \sup_\mu g^*_\mathcal{P}(\mu)+2\alpha$. For any $k$ there is a positive integer $n_k$ and a  point  $x_k$ in $X$ satisfying: 
\begin{eqnarray*}
g^*_{\mathcal{P}}(\epsilon_k)& =& \limsup_n\frac{1}{n}\sup_{x\in X}\log  \sharp \left( \mathcal{P}_n\cap B(x,n,\epsilon_k)\right),\\
&<& \frac{1}{n_k}\log \sharp \left(\mathcal{P}_{n_k}\cap B(x_k,n_k,\epsilon_k)\right)+\alpha.
\end{eqnarray*}
Moreover we may assume the sequence $(n_k)_k$ is going to infinity. 
As already observed we have also for any $p_k\in \mathcal{P}_{n_k}\cap B(x_k,n_k,\epsilon_k)$ the following inclusion: 
$$B(x_k,n_k,\epsilon_k)\subset B(p_k,n_k,2\epsilon_k).$$
We choose such a periodic point $p_k$  that its minimal period $P_k\leq  n_k$ satisfy 
$$\sharp \mathcal{P}_{P_k}\cap B(x_k,n_k,\epsilon_k)\geq \frac{1}{n_k}\sharp \mathcal{P}_{n_k}\cap B(x_k,n_k,\epsilon_k).$$ We also let $(a_l)_l$ be a nondecreasing sequence of positive integers going to infinity with $\epsilon_{a_l}\geq 2\epsilon_l$ for all large enough $l$. Then the associated  periodic measure $\nu_{p_k}$ satisfies for large $k$:

\begin{eqnarray*}
\mathfrak{P}_{a_k}(\nu_{p_k})& = & \frac{1}{P_k}\int\log \sharp \left(\mathcal{P}_{P_k}\cap B(x,P_k,\epsilon_{a_k})\right )d\nu_{p_k}(x),\\
& \geq &  \frac{1}{n_k}\int\log \sharp \left(\mathcal{P}_{P_k}\cap B(x,n_k,2\epsilon_{k})\right )d\nu_{p_k}(x),\\
&\geq & \frac{1}{n_k}\log \sharp \left(  \mathcal{P}_{n_k}\cap B(x_k,n_k,\epsilon_k) \right)-\frac{\log n_k}{n_k},\\
&> &g^*_{\mathcal{P}}(\epsilon_k)-2\alpha.
\end{eqnarray*}
We may assume the period $P_k$ of $p_k$ goes to infinity with $k$ ; otherwise for large $k$ we would have    $ g_\mathcal{P}^*(\epsilon_k)\leq 2\alpha$ and thus $g_\mathcal{P}^*\leq2\alpha$ (recall $\mathcal{P}_n$ is finite for all $n$). For  any  weak-star limit of $(\nu_{p_k})_k$, say $\mu=\lim_k\nu_{p_{\varphi(k)}}$, we have finally :
\begin{eqnarray*}
g^*_\mathcal{P}(\mu)
&\geq & \lim_l \limsup_k \mathfrak{P}_l(\nu_{p_{\varphi(k)}}),\\
&\geq &\limsup_k \mathfrak{P}_{a_{\varphi(k)}}(\nu_{p_{\varphi(k)}}),\\
&\geq & \lim_k g^*_{\mathcal{P}}(\epsilon_{\varphi(k)})-2\alpha= g^*_\mathcal{P}-2\alpha.
\end{eqnarray*}
This concludes the proof of the lemma.
\end{proof}

\newpage

j\section{Proof of Corollary \ref{ccc}}


The  statement below  is a consequence of the main result of \cite{bdo} (Theorem 5.5 and 6.5  therein). We refer to \cite{Dowb} and \cite{BD} for the notion of superenvelope of the entropy structure and its significance in the theory of symbolic extensions.

\begin{theo}\label{refo}\cite{bdo} Let $(X,T)$ be a topological dynamical system and $\mathcal{P}$  an invariant subset of $\Per(T)$ with $\sharp \mathcal{P}_n<+\infty$ for all $n$. We also assume that $X$ admits a basis of neighborhoods \footnotemark[10]
with null measure boundary for any $T$-invariant measure $\mu$ which is not a periodic measure in  $\bigcup_n \left( \Per_n\setminus \Per_n^\delta\right)$.

 Let $E$ be an affine bounded superenvelope of the entropy structure of $(X,T)$ satisfying  $E(\mu)\geq h(\mu)+g_{\mathcal{P}}^*(\mu)$ for all $T$-invariant probability measures $\mu$ (in particular $g_{\mathcal{P}}^*<+\infty$).

Then  there is  a symbolic extension
$\pi:(Y,S)\rightarrow (X,T)$ and a Borel embedding $\psi:B\rightarrow Y$ with $B$ a Borel set with full measure for every ergodic measure except  periodic measures in $\bigcup_n \left( \Per_n\setminus \mathcal{P}_n\right)$ such that $\psi\circ f=\sigma\circ \psi$ and $\pi\circ \psi=\Id_B$ and

$$\sup_{\nu,  \ S^*\nu=\nu \text{ and } \pi^*\nu=\mu}h(\nu)= E(\mu).$$

Moreover the cardinality of the alphabet of $Y$ may be chosen to be less than or equal to $e^{ \max \left( \sup_{\mu}E(\mu),\pper^\mathcal{P} \right) }+1$ with  $\pper^\mathcal{P}:=\sup_{n\in \mathbb{N}\setminus \{0\}}\frac{\log \sharp \mathcal{P}_n}{n}$.
\end{theo}

By \cite{Lin} the existence of the satisfactory basis of neighborhoods is ensured whenever $X$ is finite dimensional. 
Then it follows from the previous works on symbolic extensions for $C^r$ systems with $r>1$ given in \cite{Burguet,dowm} that  the functions 
$$\mu \mapsto h(\mu)+\frac{2(3r-1)}{r(r-1)}\left(\chi^+(\mu,f)+\chi^+(\mu,f^{-1})\right)$$
and 
$$\mu \mapsto h(\mu)+\frac{\chi^+(\mu,f)}{r-1}$$
are respectively superenvelopes of the entropy structure for a $C^r$ surface diffeomorphism  and a $C^r$ interval map  $f$ with $r>1$. Thus, together with Theorems \ref{mainn} and \ref{key} we get Corollary \ref{ccc}.\\

Finally let us explain how we deduce the above theorem from \cite{bdo}. In fact Theorem \ref{refo} corresponds to Theorem 5.5  \footnotemark[10] \footnotetext[10]{In \cite{bdo} the authors work with systems with the small boundary property. But as we only need to embedd periodic points in $\mathcal{P}$ the weaker assumption in Theorem \ref{refo} is sufficient.} in \cite{bdo} where the function $g_{\mathcal{P}}^*$ is replaced by another function 
$u_1$ defined as the limit in $k$ of $\left( (\mathfrak{Q}_k)^h \right)^{\tilde{}}$ where:
\begin{itemize}
\item  $$\mathfrak{Q}_k(\nu_n)= \frac{1}{n}\int\log \sharp \left(\mathcal{Q}_n\cap B(x,n,\epsilon_k)\right )d\nu_n(x),$$ where $\mathcal{Q}_n$ is the subset of $\mathcal{P}$ of periodic points with minimal period $n$ and $\nu_n$ is an ergodic invariant measure supported on $\mathcal{Q}_n$,
\item $(\mathfrak{Q}_k)^h$ is the harmonic 
extension of the function $\mathfrak{Q}_k$  (it was first defined  only for ergodic periodic measures) by letting $
\mathfrak{Q}_k(\nu)=0$ for ergodic non periodic measures,
\item $\left( (\mathfrak{Q}_k)^h \right)^{\tilde{}}$ is the upper semi-continuous envelope  of  $(\mathfrak{Q}_k)^h$, i.e.  the smallest upper semi-continuous function larger than $ (\mathfrak{Q}_k)^h$.
\end{itemize}

Thus we only have to show that any affine superenvelope $E$ of the entropy structure with $E-h\geq g_{\mathcal{P}}
^*=\lim_k(\mathfrak{P}_k)^{\tilde{}}$ also satisfies $E-h\geq u_1=\lim_k(\left( (\mathfrak{Q}_k)^h \right)^{\tilde{}}$. 
As the function $E-h$ is  harmonic and  $\mathfrak{P}_k\geq (\mathfrak{Q}_k)^h$  on ergodic measures this follows from 
Lemma 8.2.13   as in the proof of Lemma 8.2.14 in \cite{Dowb}. This concludes the proof of Theorem \ref{refo} from Theorem 5.5 in \cite{bdo}.

\end{large}
\end{document}